\documentclass[a4paper,11pt]{article}
\usepackage[utf8]{inputenc}
\usepackage[T1]{fontenc}
\usepackage[english]{babel}

\usepackage[tbtags]{amsmath}
\usepackage{amssymb}
\usepackage{amsfonts}
\usepackage{amsthm}
\usepackage{calrsfs}
\usepackage{array}
\usepackage{bbm}
\usepackage{stmaryrd}
\usepackage{upgreek}

\usepackage{mathtools}
\usepackage[svgnames]{xcolor}
\usepackage{hyperref}

\def\ll{{\hat L}}

\def\LL{{\hat L}}

\newcommand{\ds}{\displaystyle}
\newcommand{\e}{\varepsilon}

\renewcommand{\H}{{\cal H}}
\renewcommand\iint{\displaystyle\int_{-1}^1}

\renewcommand{\k}{k}

\renewcommand{\ll}{\hat L}

\newcommand{\m}[1]{\mathbbm{#1}}
\newcommand{\pnu}{\partial_\nu}
\newcommand{\pp}{\Pi}
\newcommand{\ps}{{\partial_s}}
\newcommand{\py}{{\partial_y}}
\newcommand{\q}[1]{\mathcal{#1}}

\newcommand\RR{{\cal R}}
\renewcommand\SS{{\cal S}}

\newcommand{\vc}[2]{\begin{pmatrix} #1\\#2\end{pmatrix}}

\def\R{{\mathbb {R}}}

\DeclareMathOperator{\argth}{\mathrm{argth}}

\theoremstyle{plain}
\newtheorem{thm}{Theorem}
\newtheorem*{thm*}{Theorem}
\newtheorem{coro}[thm]{Corollary}

\newtheorem{prop}{Proposition}[section]
\newtheorem{cor}[prop]{Corollary}
\newtheorem{lem}[prop]{Lemma}
\newtheorem{defi}[prop]{Definition}

\newtheorem{cli}[prop]{Claim}

\theoremstyle{definition}

\theoremstyle{remark}
\newtheorem*{nb}{Remark}

\makeatletter
\def\blfootnote{\xdef\@thefnmark{}\@footnotetext}
\makeatother

\title{\bf Prescribing the center of mass of a multi-soliton solution for a perturbed semilinear wave equation }

\author{M.A. Hamza\\
{\it \small 
Imam Abdulrahman Bin Faisal University
P.O. Box 1982 Dammam, Saudi Arabia}\\
Hatem Zaag\\
{\it \small Universit\'e Paris 13, Sorbonne Paris Cit\'e},\\
{\it \small LAGA, CNRS (UMR 7539), F-93430, Villetaneuse, France}
}

\date{}

\allowdisplaybreaks

\begin{document}

\maketitle

\begin{abstract}
We construct a   finite-time blow-up solution for
a class of  strongly perturbed  semilinear wave
equation with an isolated  characteristic point  in one space dimension.
    Given any integer $k\ge 2$ and $\zeta_0 \in \m R$,
we construct a blow-up solution with a characteristic point $a$, such that the asymptotic behavior of the solution near $(a,T(a))$ shows
a decoupled sum of $k$ solitons with alternate signs, whose centers (in the hyperbolic geometry) have $\zeta_0$ as a center of mass, for all times.  Although the result is similar to the unperturbed case in its statement,  our method is new. Indeed,  our perturbed equation is not invariant under  the Lorentz transform,  and this  requires new ideas. In fact, 
 the main difficulty  in this paper  is to prescribe the center of mass $\zeta_0 \in \m R$.
We would like to mention that our method is valid also in  the unperturbed case, and simplifies the original proof by
 C\^ote and Zaag \cite{CZcpam13}, as far as the center of mass prescription is concerned.
\end{abstract}

\medskip



\medskip


\section{Introduction}
This paper is concerned with blow-up solutions of a perturbed semilinear wave equation
\begin{equation}\label{PNLW}
\left\{
\begin{array}{l}
\partial^2_{t} u =\partial^2_{x} u+|u|^{p-1}u+
f(u)+g(x,t,u,\partial_x u,\partial_t u),\\
u(0)=u_0\mbox{ and }\partial_t u(0)=u_1,
\end{array}
\right.
\end{equation}
where $u(t):x\in{\m R} \rightarrow u(x,t)\in{\m R}$, $p>1$, $u_0\in \rm H^1_{\rm loc,u}$
and $u_1\in \rm L^2_{\rm loc,u}$ with
 $\|v\|_{\rm L^2_{\rm loc,u}}^2=\ds\sup_{a\in {\m R}}\int_{|x-a|<1}|v(x)|^2dx$
 and $\|v\|_{\rm H^1_{\rm loc,u}}^2 = \|v\|_{\rm L^2_{\rm loc,u}}^2+\|\nabla v\|_{\rm L^2_{\rm loc,u}}^2$.
We assume that  $f$  and $g$  are ${\cal {C}}^1$
functions, where $f:{\m R}\rightarrow {\m R} $  and $g:{\m
R}^{4}\rightarrow {\m R} $ satisfy the following conditions:
\begin{equation*}
(H)\left\{
\begin{array}{ll}
 |{f(u)}|\le
M_0\Big(1+\frac{|u|^p}{\log (2+u^2)^{\alpha}}\Big) , \  &{\textrm {for all }}\ u\in {\m R},\ \\
 |{g(x,t,u,v,z)}|\le M_0(1+|u|^{\frac{p+1}2}+|v|+|z|), & {\textrm {for all }}\ x,
t,u,v,z\in {\m R}.
\end{array}
\right.
\end{equation*}
where $M_0>0$ and $\alpha>1$.

\medskip

Under the  more restrictive assumptions
\begin{equation*}
(H')\left\{
\begin{array}{ll}
 |{f(u)}|\le
M_1(1+|u|^q) , \  &{\textrm {for all }}\ u\in {\m R},\ \\
 |{g(x,t,v,z)}|\le M_1(1+|v|+|z|), & {\textrm {for all }}\ x,
t,v,z\in {\m R}.
\end{array}
\right.
\end{equation*}
where $M_1>0$ and $q<p$,  we are able to add more information to our results (see Theorem 1 and Corollary 2).

As we will explain in details later, we would like to draw the attention of the reader to the fact that we 
are making  two major steps in considering equation \eqref{PNLW} under assumption $(H)$ in comparison with our previous  papers dedicated to perturbations of the wave equation and written under assumption $(H')$ (\cite{HZnonl12},
\cite{HZjhde12},
\cite{HZBSM}, \cite{HZdcds13})

- we allow a quasi critical perturbation $f(u)$ (in $\log$ scales), but this was already the case in 
 (\cite{H},
\cite{HS1},
\cite{HS2}).

- we allow a $u$ dependance in $g$, with a growth up to $|u|^{\frac{p+1}2}$.

\bigskip

The Cauchy problem for
equation  \eqref{PNLW} is solved   in the space ${\rm H}^1_{\rm loc,u}\times {\rm L}^2_{\rm loc, u}$.
 This follows from the finite speed of propagation and the well-posedness in $ H^1({\m R})\times L^2({\m R})$
(see for example Georgiev and Todorova \cite{GTjde94}).
 The existence of blow-up solutions $u(t)$ of  \eqref{PNLW}  follows from
energy techniques (see for example Levine and Todorova \cite{HAG} and Todorova \cite{T}).

\bigskip

Note that in this paper, we consider  a class of perturbations of the
idealized equation  (when $f\equiv g\equiv 0$).
This is quite meaningful, since physical models are sometimes damped and hardly come with a pure power source term
 (see Whitham \cite{Wjws99}). For
  more applications  in general relativity, see Donninger,  Schlag and  Soffer \cite{DSS}.

\bigskip

If $u$ is an arbitrary blow-up solution of   \eqref{PNLW}, we define (see for example Alinhac \cite{Apndeta95}) a 1-Lipschitz curve $\Gamma=\{(x,T(x))\}$
such that the maximal influence domain $D$ of $u$ (or the domain of definition of $u$) is written as
\begin{equation}\label{defdu}
D=\{(x,t)\;|\; t< T(x)\}.
\end{equation}
$\bar T=\inf_{x\in {\m R}}T(x)$ and $\Gamma$ are called the blow-up time and the blow-up graph of $u$.
A point $x_0$ is a non characteristic point
if  there are
\begin{equation}\label{nonchar}
\delta_0\in(0,1)\mbox{ and }t_0<T(x_0)\mbox{ such that }
u\;\;\mbox{is defined on }{\mathcal C}_{x_0, T(x_0), \delta_0}\cap \{t\ge t_0\}
\end{equation}
where ${\cal C}_{\bar x, \bar t, \bar \delta}=\{(x,t)\;|\; t< \bar t-\bar \delta|x-\bar x|\}$. We denote by $\RR$ (resp. $\SS$) the set of non characteristic (resp. characteristic) points.

\bigskip

In the case  $(f,g)\equiv ( 0,0)$), equation \eqref{PNLW}  reduces to the semilinear wave equation:
\begin{equation}\label{NLW}
\left\{
\begin{array}{l}
\partial^2_{t} u =\partial^2_{x} u+|u|^{p-1}u,\\
u(0)=u_0\mbox{ and }\partial_t u(0)=u_1.
\end{array}
\right.
\end{equation}

\bigskip

 In a series of papers \cite{MZjfa07}, \cite{MZcmp08}, \cite{MZajm11} and \cite{MZisol10}
  (see also the note \cite{MZxedp10}), Merle and Zaag  together with
   C\^ote and Zaag \cite{CZcpam13} give a full picture of  blow-up
  for solutions  of  equation \eqref{NLW} in one space dimension. Furthermore, in \cite{CZcpam13}, for
 any integer $k\ge 2$ and $\zeta_0 \in \m R$,   C\^ote and Zaag  construct a blow-up solution with a characteristic point $a$, such that the asymptotic behavior of the solution near $(a,T(a))$ shows
a decoupled sum of $k$ solitons with alternate signs, whose centers (in the hyperbolic geometry) have $\zeta_0$ as a center of mass, for all times.
Let us note that the invariance of equation \eqref{NLW} under  the Lorentz transform was  crucial
in \cite{CZcpam13} to prescribe the center of mass.

\bigskip

More generally,  in \cite{HZBSM},  we 
also  construct    a finite-time blow-up solution with a characteristic point in    the unperturbed case, under assupmtion $(H')$,
 showing multi-solitons and prescribing the center of mass, if $g\equiv 0$. However, when $g\not\equiv 0$, we were unable to prescribe the center of mass  of the multi-soliton, since in this case  equation \eqref{PNLW} is not invariant under  the Lorentz transform. That obstruction  justifies
our new paper, where we invent  new ideas  in order to   prescribe the center of mass  of the multi-soliton.
More importantly, our method works under the more general assumption $(H)$. Moreover,  this new method is of course valid also for equation \eqref{NLW}.
Even better, the prescription
of the center of mass for \eqref{NLW} is easier thanks to our method.

\bigskip

Our aim in this work is to    prescribe the center of mass  of the multi-soliton.
 More precisely, for
 any integer $k\ge 2$ and $\zeta_0 \in \m R$,  we construct a blow-up solution with a characteristic point $a$, such that the asymptotic behavior of the solution near $(a,T(a))$ shows
a decoupled sum of $k$ solitons with alternate signs, whose centers (in the hyperbolic geometry) have $\zeta_0$ as a center of mass, for all times.

\bigskip

Before stating our result, let us introduce  the following similarity variables, for any   $(x_0,T_0)$ such that  $0< T_0\le T(x_0)$:
\begin{equation} \label{def:w}
w_{x_0, T_0}(y,s) = (T_0 -t)^{\frac{2}{p-1}} u(x,t), \quad y = \frac{x-x_0}{T_0-t}, \quad s = - \log (T_0-t).
\end{equation}
If $T_0=T(x_0)$, we will simply write $w_{x_0}$ instead of $w_{x_0, T(x_0)}$.
The function $w=w_{x_0}$ satisfies the following equation for all $y\in(-1,1)$ and $s\ge - \log T(x_0)$:
\begin{eqnarray}
\label{eq:nlw_w}
\partial^2_{s}w&=&\q L w -\frac{2(p+1)}{(p-1)^2}w+|w|^{p-1}w
-\frac{p+3}{p-1}\partial_sw-2y\partial^2_{y,s}
w +e^{-\frac{2ps}{p-1}}f\Big(e^{\frac{2s}{p-1}}w\Big)\\
&&+
e^{-\frac{2ps}{p-1}}g\Big(x_0+ye^{-s},T_0-e^{-s},e^{\frac{(p+1)s}{p-1}}\partial_yw,e^{\frac{(p+1)s}{p-1}}(\partial_sw+y\partial_y
w+\frac{2}{p-1}w)\Big),\nonumber
\end{eqnarray}
where  $ \q L w = \frac{1}{\rho} \partial_y (\rho
(1-y^2) \py w)$  and $\rho = \rho (y)=
(1-y^2)^{\frac{2}{p-1}}.$

\bigskip

In the unperturbed case where $f\equiv g \equiv  0$,
the Lyapunov functional for equation \eqref{eq:nlw_w}
\begin{equation}\label{defenergy}
E(w(s))= \iint \left(\frac 12 \left(\partial_s w\right)^2 + \frac 12  \left(\partial_y w\right)^2 (1-y^2)+\frac{(p+1)}{(p-1)^2}w^2 - \frac 1{p+1} |w|^{p+1}\right)\rho dy
\end{equation}
is defined for $(w,\partial_s w) \in \H$ where
\begin{equation}\label{defnh0}
\H = \left\{(q_1,q_2)
\;\; \middle| \;\;\|(q_1,q_2)\|_{\H}^2\equiv \int_{-1}^1 \left(q_1^2+\left(q_1'\right)^2  (1-y^2)+q_2^2\right)\rho dy<+\infty\right\}.
\end{equation}
We also introduce the projection of the space $\H$  defined in \eqref{defnh0} on the first coordinate:
\begin{equation}\label{30oct1}
\H_0 = \left\{r\in H^1_{loc}\;\;\middle|\;\;\|r\|_{\H_0}^2\equiv \int_{-1}^1 \left(r^2+\left(r'\right)^2  (1-y^2)\right)\rho dy<+\infty\right\}.
\end{equation}

Moreover, we introduce for all $|d|<1$ the following solitons defined by
\begin{equation}\label{defkd}
\kappa(d,y)=\kappa_0 \frac{(1-d^2)^{\frac 1{p-1}}}{(1+dy)^{\frac 2{p-1}}}\mbox{ where }\kappa_0 = \left(\frac{2(p+1)}{(p-1)^2}\right)^{\frac 1{p-1}} \mbox{ and }|y|<1.
\end{equation}
Note that $\kappa(d)$ is a stationary solution of \eqref{eq:nlw_w},
in the particular case where $(f,g) \equiv (0,0)$.
We also   introduce
\begin{equation}\label{solpart}
\bar \zeta_i(s) = \left(i-\frac{(k+1)}2\right)\frac{(p-1)}2\log s + \bar\alpha_i(p,k)
\end{equation}
where the sequence $(\bar\alpha_i)_{i=1,\dots,k}$ is uniquely determined by the fact that $(\bar \zeta_i(s))_{i=1,\dots,k}$ is an explicit solution with zero center of mass for the following ODE system:
\begin{equation} \label{eq:tl}
 \frac 1{c_1}\dot \zeta_i = e^{ - \frac{2}{p-1} (\zeta_i - \zeta_{i-1}) } - e^{- \frac{2}{p-1} (\zeta_{i+1} - \zeta_i) }, \qquad \forall i=1,\dots,k.
\end{equation}
where $c_1=c_1(p)>0$ and $\zeta_0(s)\equiv \zeta_{k+1}(s) \equiv 0$ (see Section 2 in \cite{CZcpam13}  for more details).

\bigskip

Given an arbitrary blow-up solution $u(x,t)$ to \eqref{PNLW} and a characteristic point
$x_0$, we could extend in  \cite{HZBSM} the results first proved for  \eqref{NLW} by Merle and Zaag in
 \cite{MZbsm11}, and showed
%
the following results, under the assumption $(H')$ (and we think  the results can be easily extended to assumption   $(H)$ following our strategy in  \cite{HZBSM})

\medskip

\newblock{
{\bf (Description of the behavior of $w_{x_0}$ where $x_0$ is
characteristic)} Consider $x_0\in \SS$. Then, there is $\zeta_0(x_0)\in \m R $
such that
\begin{equation}\label{cprofile00}
\left\|\vc{w_{x_0}(s)}{\ps w_{x_0}(s)} - \theta_1\vc{\ds\sum_{i=1}^{k(x_0)} (-1)^{i+1}\kappa(d_i(s))}0\right\|_{\H} \to 0
\end{equation}
as $s\to \infty$, for some $k(x_0)\ge 2,$  $\theta_1=\pm 1$
 and continuous $d_i(s) = -\tanh \zeta_i(s)$ with
\begin{equation}\label{equid0}
\zeta_i(s) = \bar \zeta_i(s) + \zeta_0,
\end{equation}
where $\bar \zeta_i(s)$ is introduced above in \eqref{solpart}.
}


\medskip

Following this \newblock{ description} result, we naturally ask the corresponding
\newblock{ construction} question:

\medskip

 \newblock{Can we construct an example for each blow-up modality described in \eqref{cprofile00}?}

\medskip

 In other words, given $k\ge 2$ and and $\zeta_0\in \m R$, can we construct a solution to equation \eqref{PNLW} showing the behavior \eqref{cprofile00} and \eqref{equid0}?

\medskip

In   \cite{CZcpam13},  C\^ote and Zaag showed that this  is possible in the  unperturbed case \eqref{NLW}.
  In \cite{HZBSM}, we had a partial result for equation \eqref{PNLW}, in the sense that we could prescribe the number of solitons, but not the location of their center of mass, 
  unless $g\equiv 0$. Indeed, the method of  \cite{CZcpam13} extends with no difficulty to equation \eqref{PNLW},   if  $g\equiv 0$. If not, or if assumption $(H)$ holds,  we need new ideas and this is the aim of the paper.
More precisely, this is our main result:
\begin{thm}[\bf{Existence of a solution  to \eqref{PNLW} with prescribed blow-up behavior at a characteristic point}]\label{mainth}
Assume that assumption $(H)$ holds.  For any integer $k\ge 2$ and $\zeta_0\in \m R$, there exists a blow-up solution $u(x,t)$ to equation \eqref{PNLW} in $\rm H^1_{\rm loc,u}\times \rm L^2_{\rm loc,u}(\m R)$ 
such that
\begin{equation}\label{cprofile0}
\left\|\vc{w_0(s)}{\ps w_0(s)} - \vc{\ds\sum_{i=1}^{k} (-1)^{i+1}\kappa(d_i(s))}0\right\|_{\H} \to 0\mbox{ as }s\to \infty,
\end{equation}
with
\begin{equation}\label{refequid1}
d_i(s) = -\tanh \zeta_i(s), \quad
\zeta_i(s) = \bar \zeta_i(s) + \zeta_0
\end{equation}
and $\bar \zeta_i(s)$ defined in \eqref{solpart}. Moreover, if $(H')$ holds, then the origin is a  characteristic point (i.e.  $0\in\SS$).
\end{thm}
\begin{nb}
Note from \eqref{refequid1} and \eqref{solpart} that the center of mass  of $\zeta_i(s)$ is fixed, in the sense that
\begin{equation}\label{barycenter}
\frac{\zeta_1(s)+ \dots +\zeta_k(s)}k= \frac{\bar\zeta_1(s)+ \dots +\bar\zeta_k(s)}k+\zeta_0=\zeta_0,\;\;\forall s\ge -\log T(0).
\end{equation}
\end{nb}

\begin{nb}
Let us remark from the blow-up description we made in  \cite{HZBSM} that,  under the assumption $(H')$
 (and we think  the results can be easily extended to assumption   $(H)$), 
 the origin is an isolated characteristic point. We also would like to  mention that,
as in the unperturbed case \eqref{NLW},    under the assumption $(H)$
 (even under assumption   $(H')$) we are unable to say whether this solution has other characteristic points or not. In particular, we have been unable to find a solution with $\SS$ exactly equal to $\{0\}$. 
Nevertheless, let us remark that from the finite speed of propagation, we can prescribe more characteristic points, as follows:
\end{nb}

\begin{coro}[Prescribing more characteristic points]\label{cormore} 
Assume that assumption $(H)$ holds.
Let $I=\{1,...,n_0\}$ or $I=\m N$ and for all $n\in I$, $x_n\in \m R$, $T_n>0$, $k_n \ge 2$ and $\zeta_{0,n}\in \m R$ such that
\begin{equation*}
x_n+T_n<x_{n+1}-T_{n+1}.
\end{equation*}
Then, there exists a blow-up solution $u(x,t)$ of equation \eqref{PNLW} in $\rm H^1_{\rm loc,u}\times \rm L^2_{\rm loc,u}(\m R)$ with $T(x_n)=T_n$ and for all $n\in I,$
\begin{equation*}
\left\|\vc{w_{x_n}(s)}{\ps w_{x_n}(s)} - \vc{\ds\sum_{i=1}^{k_n} (-1)^{i+1}\kappa(d_{i,n}(s))}0\right\|_{\H} \to 0\mbox{ as }s\to \infty,
\end{equation*}
with
\begin{equation*}
\forall i=1,\dots,k_n,\;\;d_{i,n}(s) = -\tanh \zeta_{i,n}(s),\;\;
\zeta_{i,n}(s) = \bar \zeta_i(s) + \zeta_{0,n}
\end{equation*}
and $\bar \zeta_i(s)$ defined in \eqref{solpart}. Moreover, if   $(H')$ holds, then $\{x_n\;|\; n\in I\} \subset \SS$.
\end{coro}
\begin{nb}
Again, we are unable to construct a solution with $\SS = \{x_n\;|\; n\in I\}$.
\end{nb}

As one can see from \eqref{cprofile0} and \eqref{refequid1}, the solution we have just constructed in Theorem \ref{mainth} behaves like the sum of $k$ solitons as $s\to \infty$. In the literature, such a solution is called a {\it multi-soliton solution}. Constructing multi-soliton solutions is an important problem in nonlinear dispersive equations. It has already be done for the $L^2$ critical and subcritical nonlinear Schr\"odinger equation (NLS) (see Merle \cite{Mcmp90} and Martel and Merle \cite{MMihp06}), the $L^2$ critical and subcritical generalized Korteweg de Vries equation (gKdV) (see Martel \cite{Majm05}), and for the $L^2$ supercritical case both for (gKdV) and (NLS) equations in C\^ote, Martel and Merle \cite{CMMrmia11}.

\medskip
%
%

\medskip

More generally, constructing a solution to some 
PDE
with a prescribed behavior (not necessarily multi-solitons solutions) is an important question. That question was solved 
for (gKdV) by C\^ote \cite{Cjfa06,Cdmj07}, and also for parabolic equations exhibiting blow-up, like the semilinear heat equation by
 Bressan \cite{Biumj90, Bjde92} (with an exponential source), Merle \cite{Mcpam92}, Bricmont and Kupiainen \cite{BKnonl94},
 Merle and Zaag in \cite{MZdmj97, MZcras96}, Schweyer \cite{Sjfa12} (in the critical case), Mahmoudi, Nouaili and Zaag \cite{MNZna16} (in the periodic case),
 the complex Ginzburg-Landau equation by Zaag \cite{Zihp98},  Masmoudi and Zaag in \cite{MZjfa08} and also Nouaili and Zaag  \cite{NZarma18}, 
a complex heat equation with no gradient structure by Nouaili and Zaag \cite{NZcc14}, a gradient perturbed heat equation in the subcritical case by
 Ebde and Zaag in \cite{EZsema11}, then 
by Tayachi and Zaag in the critical case in \cite{TZ15, TZnormandie15}
and also by  Ghoul,  Nguyen  and Zaag in  \cite{GNZjde17},  a strongly perturbed 
 heat equation in Nguyen and Zaag \cite{NZcs14},  a non scaling invariant  heat equation in Duong, Nguyen, and Zaag
\cite{DNZtjm18},  two non variational parabolic system 
 Ghoul,  Nguyen  and Zaag  \cite{GNZIHP18, GNZJDE18, GNZSMT18}  or a higher order parabolic  equation in \cite{GNZ18P}.
 Other examples are available for Schr\"odinger maps (see Merle, Rapha\"el and Rodnianski \cite{MRRim13}),
 the wave maps (see Ghoul, Ibrahim and Nguyen  \cite{GINpre17}), 
 and also for the Keller-Segel model (see Rapha\"el and Schweyer \cite{RSma14}, and also Ghoul and Masmoudi \cite{GM16}).

\medskip

Surprisingly enough, in both the parabolic equations above and the supercritical dispersive equations treated in \cite{CMMrmia11}, the same topological argument is crucial to control the directions of instability. This will be the case again for the semilinear wave equation \eqref{PNLW} under consideration. 
As for the strategy of the proof, we rely on  the following two step method:

\medskip

- As in the unperturbed case, thanks to a dynamical system formulation, we show that controlling the similarity variables version $w(y,s)$  solution of \eqref{eq:nlw_w} around the expected behavior \eqref{cprofile0} reduces to the control of the unstable directions, whose number is finite. This dynamical system formulation is essentially the same as the one used in    \cite{CZcpam13} and \cite{MZisol10}.

- Then, we solve the finite dimensional problem thanks to a topological argument based on index theory. 

\medskip

We would like to insist on the fact that we introduce a new feature in the method to allow us 
to directly
 prescribe the center of mass as required in \eqref{barycenter}.

\section{Construction of a multi-soliton solution in similarity variables}\label{secweak}
In this section,  we assume that  $(H)$ holds, and 
we construct a multi-soliton solution in similarity variables for   equation 
\eqref{eq:nlw_w}. At the end of the section, we use this construction to prove Theorem 1 (note that we 
don't prove Corollary 2, as it follows immediately from  Theorem 1  thanks to the 
finite speed of propagation). At the very end of the section, we will consider the more restrictive hypothesis $(H')$
to show that the origin is a characteristic point.

\medskip

Technically, we use the dynamical system formulation introduced in \cite{MZisol10} and used in \cite{CZcpam13}.
 For that purpose, we introduce for all $d\in (-1,1)$ and $\nu >-1+|d|$, $\kappa^*(d,\nu,y) = (\kappa_1^*, \kappa_2^*)(d,\nu,y)$ where
\begin{eqnarray}\label{defk*}
\kappa_1^*(d,\nu, y) & =&
\kappa_0\frac{(1-d^2)^{\frac 1{p-1}}}{(1+dy+\nu)^{\frac 2{p-1}}},\label{defk*} \\
\kappa_2^*(d,\nu, y) & =& \nu \pnu \kappa_1^*(d,\nu, y) =
-\frac{2\kappa_0\nu}{p-1}\frac{(1-d^2)^{\frac 1{p-1}}}{(1+dy+\nu)^{\frac {p+1}{p-1}}}.\label{defk**}
\end{eqnarray}
We refer to these functions as ``generalized solitons'' or solitons for short. Notice that for any $\mu\in\m R$, $\kappa^*(d,\mu e^s,y)$ is a solution to equation \eqref{eq:nlw_w} in the unperturbed case. Then, note that:\\
- $\kappa^*(d,\mu e^s,y)\to (\kappa(d),0)$ in $\H$ as $s\to -\infty$;\\
- when $\mu=0$, we recover the stationary solutions $(\kappa(d),0)$ defined in \eqref{defkd};\\
- when $\mu>0$, the solution exists for all $(y,s) \in (-1,1)\times \m R$ and converges to $0$ in $\H$ as $s\to \infty$ (it is a heteroclinic connection between $(\kappa(d),0)$ and $0$);\\
- when $\mu<0$, the solution exists for all $(y,s) \in (-1,1)\times \left(-\infty, \log\left(\frac {|d|-1}\mu\right)\right)$ and blows up at time $s=\log\left(\frac {|d|-1}\mu\right)$.\\
We also introduce for $l=0$ or $1$, for any $d\in (-1,1)$ and $r\in \H$,
\begin{gather}
\pp_l^d(r) =\phi\left(W_l(d), r\right), \quad \text{where} \label{defpdi} \\
\begin{aligned}
\phi(q,r) & := \int_{-1}^1 \left(q_1r_1+q_1' r_1' (1-y^2)+q_2r_2\right)\rho dy, \nonumber\\
W_l(d,y) & := (W_{l,1}(d,y), W_{l,2}(d,y)), \quad \text{with} \nonumber
\end{aligned} \\
W_{1,2}(d,y)(y)= c_1(d)\frac{(1-d^2)^{\frac 1{p-1}}(1-y^2)}{(1+dy)^{\frac 2{p-1}+1}},\quad
W_{0,2}(d,y) = c_0\frac {(1-d^2)^{\frac 1{p-1}}(y+d)}{(1+dy)^{\frac 2{p-1}+1}}, \label{defWl2-0}
\end{gather}
for some positive $c_1(d)$ and $c_0$, and $W_{l,1}(d,y)\in \H_0$ is uniquely determined as the solution of
\begin{equation}\label{eqWl1-0}
-\q L r + r =\left(l - \frac{p+3}{p-1}\right)W_{l,2}(d) - 2 y\py W_{l,2}(d)+ \frac 8{p-1} \frac{W_{l,2}(d)}{1-y^2}
\end{equation}
normalized by the fact that $\pp_l^d(F_l(d)) =\phi\left(W_l(d), F_l(d)\right)$, where
\begin{equation*}
F_1(d,y)=(1-d^2)^{\frac p{p-1}}\vc{\frac{1}{(1+dy)^{\frac 2{p-1}+1}}}{\frac{1}{(1+dy)^{\frac 2{p-1}+1}}},\;\; F_0(d,y)=(1-d^2)^{\frac 1{p-1}}\vc{\ds\frac{y+d}{(1+dy)^{\frac 2{p-1}+1}}}{0}
\end{equation*}
(see estimate (3.57) in \cite{MZajm11} for more details).

\bigskip

Given $k\ge 2$, $\zeta_0\in  {\m R}$  and $s_0>0$, we will construct the multi-solution as a solution
to the Cauchy problem of equation \eqref{eq:nlw_w} with initial data (at $s=s_0$)  depending on $k+1$ parameters:
  $|\nu_{i,0}|\le s_0^{-\frac 12 - |\gamma_i|}$ for $i=1,\dots,k,$ and 
$|\phi_{1,0}|\le s_0^{-\eta},$
 given by
\begin{equation}\label{w0}
(w(y,s_0),\partial_sw(y,s_0))=\sum_{i=1}^k (-1)^i \kappa^*\left(\hat d_i(s_0, \phi_{1,0}),\nu_{i,0},y
\right), 
\end{equation}
where 
\begin{align}
\hat d_i(s_0, \phi_{1,0})= -\tanh (\bar \zeta_i(s_0)+\zeta_0+\frac{p-1}{2} \phi_{1,0}),\label{v14}
\end{align}
$\bar \zeta_i(s_0)$ is defined in \eqref{solpart}, $\gamma_i$ is defined by
\begin{equation}\label{defgi}
\ds \gamma_i = (p-1) \left(-i + \frac{k+1}{2} \right),
\end{equation}
and the constant  $\eta>0$ will be defined later in \eqref{defeta}.

\medskip

 In comparison 
with the unperturbed case treated in \cite{CZcpam13},  we have one more parameter here ($\phi_{1,0}$), and this 
parameter is precisely the one that will allow us to prescribe the center of mass. Note also that in \cite{CZcpam13}, the corresponding construction was done with 
$\zeta_0=0$. Accordingly  estimate \eqref{v14} in \cite{CZcpam13} was satisfied with $\zeta_0=0$ and 
$\phi_{1,0}=0$. Here lays a major difference between our approach and that of \cite{CZcpam13}, in the sense that we construct our solution in relation with $(\bar \zeta_i(s_0)+\zeta_0)_i$, which is a particular solution of 
\eqref{eq:tl}, whereas in \cite{CZcpam13}, the construction is done only for  $\zeta_0=0$.

\medskip

 Such a solution will be denoted by
 $w(s_0,(\nu_{i,0})_i,\phi_{1,0},y,s)$, or when there is no ambiguity, by $w(y,s)$ or $w(s)$ for short. We will show that when $s_0$ is fixed large enough, we can fine-tune the parameters 
$$\nu_{i,0} \in \Big[-s_0^{-\frac 12 - |\gamma_i|},s_0^{-\frac 12 - |\gamma_i|}\Big]\quad
 (\textrm{for}\ \  i=1,\dots,k)\qquad
\textrm{
and}\qquad
\phi_{1,0}\in \Big[-s_0^{ - \eta},s_0^{-\eta}\Big],$$
so that the solution $w(s_0,(\nu_{i,0})_i,\phi_{1,0},y,s)$ 
will be decomposed as a sum of $k$ decoupled solitons. This is the aim of the section.\\

%
%
%


\begin{prop}[A multi-soliton solution in the $w(y,s)$ setting]\label{propw} For any integer $k\ge 2$, and $\zeta_0 \in  {\m R}$
there exist $s_0>0$, $\nu_{i,0}\in\m R$ for $i=1,\dots,k$ and $\phi_{1,0}\in\m R$ such that equation \eqref{eq:nlw_w} with initial data (at $s=s_0$) given by \eqref{w0} is defined for all $(y,s)\in(-1,1)\times [s_0, \infty)$, satisfies $(w(s), \ps w(s))\in\H$ for all $s\ge s_0$, and
\begin{equation}\label{cprofile}
\left\|\vc{w(s)}{\ps w(s)} - \vc{\ds\sum_{i=1}^{k} (-1)^{i+1}\kappa(d_i(s))}0\right\|_{\H} \to 0\mbox{ as }s\to \infty,
\end{equation}
for some continuous $d_i(s) = -\tanh \zeta_i(s)$ satisfying
\begin{equation}\label{equid}
\zeta_i(s) -\bar \zeta_i(s)\to\zeta_0\mbox{ as }s\to\infty\mbox{ for }i=1,\dots,k
\end{equation}
where the $\bar\zeta_i(s)$ are introduced in \eqref{solpart}.
\end{prop}

\begin{nb}
Note from \eqref{w0} that initial data are in $H^1\times L^2(-1,1)$.
Going back to the $u(x,t)$ formulation, we see that initial data is also in $H^1\times L^2(-1,1)$ of the initial section of the backward light-cone. Therefore, from the solution to the Cauchy-problem in light-cones, we see that the solution stays in $H^1\times L^2$ of any section.
\end{nb}

\bigskip

As one can see from \eqref{w0}, at the initial time $s=s_0$, $w(y,s_0)$ is a pure sum of solitons. From the continuity of the flow associated with equation \eqref{eq:nlw_w} in $\H$ (this continuity comes from the continuity of the flow associated with equation \eqref{PNLW} in $H^1\times L^2$ of sections of backward light-cones), $w(y,s)$ will stay close to a sum of solitons, at least for a short time after $s_0$. In fact, we can do better, and impose some orthogonality conditions, killing the zero and expanding directions of the linearized operator of equation \eqref{PNLW} around the sum of solitons. The following modulation technique from Merle and Zaag in \cite{MZisol10} is crucial for that.
\begin{prop}[A modulation technique; Proposition 3.1 of \cite{MZisol10}]\label{lemode0}
For all $A\ge 1$, there exist $E_0(A)>0$ and $\epsilon_0(A)>0$ such that for all  $E\ge E_0$ and $\epsilon\le \epsilon_0$,  if $v\in \H$ and for all $i=1,\dots,k$, $(\hat d_i,\hat \nu_i)\in(-1,1)\times \m R$ are such that
\begin{equation*}
-1+\frac 1A \le \frac{\hat \nu_i}{1-|\hat d_i|}\le A,\;\;
\hat \zeta_{i+1}^*-\hat \zeta_i^*\ge E\mbox{ and }\|\hat q\|_{\H}\le \epsilon
\end{equation*}
where $\hat q = v-\ds\sum_{j=1}^{\k}(-1)^j \kappa^*(\hat d_j,\hat \nu_j)$ and $\ds{\hat d_i^*=\frac{\hat d_i}{1+\hat \nu_i} = -\tanh \hat \zeta_i^*}$, then, there exist $(d_i,\nu_i)$ such that for all $i=1,\dots,\k$ and $l=0,1$,
\begin{enumerate}
\item $\pp_l^{d_i^*}(q)=0$ where $q:=v-\sum_{j=1}^{\k}(-1)^j \kappa^*(d_j,\nu_j)$,
\item $\ds \left| \frac{\nu_i}{1-|d_i|}- \frac{\hat \nu_i}{1-|\hat d_i|}\right|+|\zeta_i^*-\hat \zeta_i^*|\le C(A)\|\hat q\|_{\H}\le C(A)\epsilon$,
\item $\ds -1+\frac 1{2A} \le \frac{\nu_i}{1-|d_i|}\le A+1$,
$\ds \zeta_{i+1}^*-\zeta_i^*\ge \frac E2$ and $\|q\|_{\H} \le C(A)\epsilon$,
\end{enumerate}
where $d_i^*=\frac{d_i}{1+\nu_i} = -\tanh \zeta_i^*$.
\end{prop}

Let us apply
 this proposition with $v=(w(y,s_0),\partial_s w(y,s_0))$ defined in  \eqref{w0}, $\hat d_i=\hat d_i(s_0, \phi_{1,0})$  defined in \eqref{v14} and $\hat \nu_i=\nu_{i,0}$. Clearly, we have $\hat q=0$. Then, from \eqref{w0}, \eqref{solpart} and straightforward calculations, we see that
\[
\frac{|\hat \nu_i|}{1-|\hat d_i|}\le \frac C{\sqrt{s_0}} \quad \text{and} \quad \hat \zeta_{i+1}^*-\hat \zeta_i^*\ge \frac{p-1}4 \log s_0
\]
for $s_0$ large enough. Therefore, Proposition \ref{lemode0} applies with $A=2$ and from the continuity of the flow associated with equation \eqref{eq:nlw_w} in $\H$, we have a maximal $\bar s=\bar s(s_0,(\nu_{i,0})_i, \phi_{1,0})>s_0$ such that $w$ exists for all time $s\in [s_0, \bar s)$ and $w$ can be modulated
in the sense that
\begin{equation}\label{defq}
(w(y,s),\partial_s w(y,s) ) = \sum_{i=1}^k (-1)^i \kappa^*(d_i(s), \nu_i(s))+q(y,s)
\end{equation}
where the parameters $d_i(s)$ and $\nu_i(s)$ are such that for all $s\in[s_0, \bar s]$,
\[
\pp_l^{d_i^*(s)}(q(s)) =0,\;\;\forall l=0,1,\;\;i=1,\dots,k
\]
and
\begin{equation}\label{conmod}
\frac{|\nu_i(s)|}{1-|d_i(s)|}\le s_0^{-1/4},\;\;
\zeta_{i+1}^*(s)-\zeta_i^*(s)\ge \frac{(p-1)}8 \log s_0\quad \text{and} \quad \|q(s)\|_{\H}\le \frac 1{\sqrt{s_0}},
\end{equation}
where $\zeta_i^*(s)=- \arg\tanh \big(\frac{d_i(s)}{1+\nu_i(s)}\big)$.

\medskip

Two cases then arise:

- either $\bar s(s_0,(\nu_{i,0})_i, \phi_{1,0})=+\infty$;

- or $\bar s(s_0,(\nu_{i,0})_i, \phi_{1,0})<+\infty$ and one of the $\le$ symbols in \eqref{conmod} is a $=$.

At this stage, we see that controlling the solution $w(s)\in \q H$ is equivalent to controlling $q\in \q H$, $(d_i(s))_i\in(-1,1)^k$ and $(\nu_i(s))_i\in\m R^k$. Introducing
\begin{equation}\label{defJ}
\zeta_i(s)=- \arg\tanh (d_i(s)), \qquad
J = \sum_{i=1}^{k-1} e^{-\frac{2}{p-1} (\zeta_{i+1} - \zeta_{i})}, \quad \bar J = \sum_{i=1}^k \frac{|\nu_i|}{1-d_i^2}, 
\end{equation}
we  project equation \eqref{eq:nlw_w} to derive 
the following estimates: 
\begin{prop}[Dynamics of the parameters]\label{propdyn}
There exists $\delta>0$
such that for $s_0$ large enough and for all $s\in[s_0,\bar s)$, we have
\begin{align}
\frac{| \dot \nu_i - \nu_i |}{1-d_i^2} & \le C  \left( \| q \|_{\q H}^2 + J + \| q \|_{\q H} \bar J\right)+\frac{C}{s^{\alpha}},
\label{est:nu}
\end{align}
\begin{align}
\left|\frac {\dot \zeta_i }{c_1} - (e^{-\frac{2}{p-1}  (\zeta_{i}-\zeta_{i-1})} - e^{-\frac{2}{p-1}  (\zeta_{i+1}-\zeta_{i})} )\right|  \le C(\| q \|_{\q H}^2 + (J +\| q \|_{\q H}) \bar J + J^{1+\delta})+\frac{C}{s^{\alpha}}, \label{est:zeta}
 \end{align}
 \begin{align}
  \qquad \qquad  \| q(s) \|_{\q H}^2 & \le C e^{-\delta(s-s_0)} \| q(s_0) \|_{\q H}^2 + C  J^{\bar p}+\frac{C}{s^{\alpha}}, \qquad \qquad  \qquad\label{est:q}
\end{align}
with
\begin{equation}\label{defpb}
 \bar p =
\begin{cases}
p & \text{ if } p <2, \\
2 - 1/100 & \text{ if } p =2, \\
2 & \text{ if } p >2,
\end{cases}
\end{equation}
where $\zeta_i(s) = - \arg\tanh d_i(s)$, $c_1=c_1(p)>0$, $J$  and $\bar J,$  are introduced in \eqref{defJ}.
\end{prop}
\begin{proof} As in the unperturbed case, this statement is devoted to understanding  the dynamics of equation \eqref{eq:nlw_w}
  near the sum of $k$ solitons.   Most of the estimates are the same as in the unperturbed case 
\eqref{NLW}  treated in \cite{CZcpam13} and \cite{MZisol10}
  and some others are  more delicate.  
 For that reason, we leave the proof to Appendix \ref{appdyn}. 
\end{proof}

\begin{nb}
Let us mention that,  the estimates \eqref{est:nu}
 \eqref{est:zeta} and  \eqref{est:q}  are  similar to the ones obtained in  the unperturbed case treated 
in \cite{CZcpam13} except for the following:

 - The   presence of the additional term $\displaystyle{ \frac{C}{s^{\alpha}}}$ which is natural  to control the perturbation  terms related to $f$ and $g$.

-  We end up, in estimate    \eqref{est:q},  we have a new term  $ J^{\bar p}$ instead of  $\hat J^2$
 (where   
$\displaystyle{ \hat J = \sum_{i=1}^{k-1} e^{- \frac{\bar p}{p-1}(\zeta_{i+1} - \zeta_{i})}}$)  in the unperturbed case. 
This   is based  on the fact that all  the  norms on $\R^n$
 are equivalent.   
\end{nb}

\medskip

In order to prove Proposition \ref{propw},
our aim is to show the existence of a solution with 
$\| q(s) \|_{\q H}\to 0,$
$\zeta_i(s) -\bar \zeta_i(s)- \zeta_0\to 0$ and  $\bar J(s)\to 0$  as $s\to \infty.$  Hence, it is natural to do as in Section 2 in \cite{CZcpam13} and linearize system \eqref{est:zeta} around $(\bar \zeta_i(s)+\zeta_0)_i$ by introducing
\begin{equation}\label{defxi1}
 \xi_i(s) = \frac 2{p-1}(\zeta_i(s) - \bar \zeta_i(s)-\zeta_0).
 \end{equation}
This is reasonable, since we see from   \eqref{est:zeta} and \eqref{conmod} that 
 $( \zeta_i(s))_i$ satisfies  a perturbed version of the system \eqref{eq:tl} satisfied by $(\bar \zeta_i(s)+\zeta_0)_i$ and studied in  \cite{CZcpam13}.

 \medskip

Following \eqref{defxi1}, if $\boldsymbol{\xi}(s)=(\xi_1(s),\dots, \xi_k(s))$, then, as it was done in the unperturbed case   in \cite{CZcpam13},  by writing  a  Taylor expansion of \eqref{est:zeta},
%
we obtain the following differential inequalities, for all $s\in [s_0, \bar s)$, and $i=1,\dots,k,$
\begin{equation}\label{eqlin1}
\left|\dot {\boldsymbol \xi} - \frac 1s M \boldsymbol \xi \right|\le \frac Cs |\boldsymbol\xi |^2+C(\| q \|_{\q H}^2 + (J +\| q \|_{\q H}) \bar J + J^{1+\delta})+\frac{C}{s^{\alpha}},
\end{equation}
where the self-adjoint $k\times k$ matrix $M = (m_{i,j})_{(i,j) \in \llbracket 1,k \rrbracket}$ is defined by 
\begin{equation}\label{defM}
m_{i,i-1}=\sigma_{i-1},\quad m_{i,i} = - (\sigma_{i-1} + \sigma_i),\quad m_{i,i+1}=\sigma_i,\qquad  m_{i,j}=0 \text{ if }|i-j|\ge 2,
\end{equation}
with
\begin{equation}\label{defsi}
\sigma_i=\frac{i(k-i)}2.
\end{equation}
Note that  the matrix $M$ is diagonalizable, with real eigenvalues
 $(-m_i)_i$,  
defined by:
\begin{equation}\label{defmi}
-m_i \equiv  -\frac{i(i-1)}2,\mbox{ for } i=1,\dots,k
\end{equation}
and the associated eigenvectors $\boldsymbol{e}_i$ normalized for the $\ell^\infty$ norm.
Note that the kernel of $M$  is  spanned   by the vector   
\begin{equation}\label{385}
{e}_1 = {}^t(1,\dots,1).
\end{equation}

\bigskip

 It is then natural to work in the basis defined by its eigenvectors $(\boldsymbol e_i)_i$ and to introduce
$\boldsymbol{\phi}(s) = (\phi_1(s), \dots, \phi_k(s))$ defined by
\begin{equation}\label{defphi}
\boldsymbol \xi(s) = \sum_{i=1}^k \phi_i(s) \boldsymbol e_i.
\end{equation}
Note that if we project the differential inequalities \eqref{eqlin1} on the 
 eigenfunctions 
  $(\boldsymbol e_i)_i$  of $M$, then we trivially obtain the following:
\begin{cor}[Dynamics for $\phi_i$]\label{cordynphi}
 For all $s\in [s_0, \bar s)$,  and $i=1,\dots,k,$
\begin{equation}
\left| \dot \phi_i + \frac{m_i}{s} \phi_i \right| \le  \frac{C}{s} \sum_{j=1}^k \phi_j^2 + C \left( \| q \|_{\q H}^2 + (J +\| q \|_{\q H}) \bar J + J^{1+\delta} \right)+\frac{C}{s^{\alpha}}.
\label{est:phi}
\end{equation}
\end{cor}
\begin{nb}
This corollary is  trivial. Indeed, we just  write the differential inequalities \eqref{eqlin1} on the basis  $(\boldsymbol e_i)_i$.
 Note that in  \cite{CZcpam13}, the authors prove  the following sharper version of \eqref{est:phi}:
\begin{equation}
\left| \dot \phi_i + \frac{m_i}{s} \phi_i \right| \le  \frac{C}{s} \sum_{j=2}^k \phi_j^2 + C \left( \| q \|_{\q H}^2 + (J +\| q \|_{\q H}) \bar J + J^{1+\delta} \right)+\frac{C}{s^{\alpha}}.
\label{est:phib}
\end{equation}
Note that in \eqref{est:phib}, the index runs from $2$ to $k$, and not from $1$ to $k$ like in  \eqref{est:phi}. That improvement was crucial in   \cite{CZcpam13}, since it was required that 
$\phi_i(s)\to 0$ as $s\to \infty$,  only for $i\ge 2$,  and not  for $\phi_1(s)$.  By the way, the proof of 
\eqref{est:phib}  is far from being easy. Here, since we work with any 
$\zeta_0 \in \m R $ (see \eqref{v14})
and aim at prescribing the center of mass, we don't need to be that accurate, and from this point of view, our proof is more simple than the proof of \cite{CZcpam13}.

%
%
\end{nb}

Note that thanks to all these changes of variables, controlling $w$ is equivalent to the control of $(q, \boldsymbol{\phi},(\nu_i)_i)$. 
Now, in order to control $w$ near multi-solitons, we introduce the following set:

\begin{defi}[Definition of a shrinking set for the parameters]\label{defVa}
We say that
$w(s)\in \q V(s)$ if and only if
\begin{equation}\label{defV}
\begin{array}{>{\ds}l}
s^{1/2+\eta} \| q \|_{\q H} \le 1,\\
 \quad \forall i =1, \dots, k, \quad s^{1/2 + |\gamma_i|} |\nu_i|  \le 1,
 \quad \text{and} \quad s^{\eta} |\phi_i |  \le 1,
\end{array},
\end{equation}
where
\begin{equation}\label{defeta}
\eta = \frac{1}{4} \min \left\{1,\delta, \frac{\bar p}{2} - \frac{1}{2},\frac{\alpha-1}2 \right\},
\end{equation}
$\delta>0$ is defined in Proposition \ref{propdyn}  and $\bar p$ is defined in \eqref{defpb}.
\end{defi}

\begin{nb}
In  \cite{CZcpam13}, the condition $ s^{\eta} |\phi_i(s) |  \le 1$ was required only for $i\ge 2$. When $i=1$, the approach of  \cite{CZcpam13}  requires only the smallness of $\phi_1(s)$ (namely that $|\phi_1(s)|\le s_0^{-\eta}$
with $s_0$ and not $s$) with no need to have   $  \phi_1(s) \to 0$ as $s\to \infty$. Here, we will get 
$  \phi_1(s) \to 0$  as $s\to \infty$, which is the key to prescribe the center of mass.
\end{nb}

\bigskip

From the existence of $\bar s$ (defined right before \eqref{defq}), we know that there is a maximal $s^*(s_0,(\nu_{i,0})_i, \phi_{1,0})\in [s_0, \bar s)$ such that for all $s\in[s_0, s^*)$, $w(s) \in \q V(s)$ and:\\
- either $s^*=+\infty$,\\
- or $s^*<+\infty$ and from continuity, $w(s^*) \in \partial \q V(s^*)$, in the sense that one $\le$ symbol in \eqref{defV} has to be replaced by the $=$ symbol.

\medskip

Our aim is to show that for $s_0$ large enough, one can find a parameter 
$\big((\nu_{i,0})_i,\phi_{1,0}\big)$ in $$\prod_{i=1}^k[-s_0^{-\frac 12 - |\gamma_i|},s_0^{-\frac 12 - |\gamma_i|}]  \times  \Big[-s_0^{ - \eta},s_0^{-\eta}\Big]$$
such that
\begin{equation}\label{aim}
s^*(s_0,(\nu_{i,0})_i, \phi_{1,0})=+\infty.
\end{equation}

With 
 Proposition \ref{propdyn} and Corollary \ref{cordynphi} at hand, we are in a position to prove the following, which directly implies Proposition \ref{propw}:
\begin{prop}[A solution $w(y,s)\in \q V(s)$]\label{reducw} For $s_0$ large enough, there exists $$(\nu_{i,0})_i\in\prod_{i=1}^k[-s_0^{-\frac 12 - |\gamma_i|},s_0^{-\frac 12 - |\gamma_i|}]\qquad  \textrm{and}\qquad
\phi_{1,0}\in \Big[-s_0^{ - \eta},s_0^{-\eta}\Big]$$ such that equation \eqref{eq:nlw_w} with initial data (at $s=s_0$) given by \eqref{w0} is defined for all $(y,s)\in(-1,1)\times [s_0, \infty)$ and satisfies $w (s) \in \q V(s)$ for all $s \ge s_0$.
\end{prop}
\begin{proof}[Proof of Proposition \ref{reducw}]
Let $s_0$ be large enough.
Define $\m B$ (resp. $\m S$) the unit ball (resp. sphere) in $(\m R^{k+1}, \ell^\infty)$, and the rescaling function
\begin{equation}\label{rescl}
\Gamma_s:{}^t (\boldsymbol{ \nu_1}, \dots, \boldsymbol{ \nu_k},\boldsymbol{ \phi_{1,0}}) \mapsto {}^t ( s^{-1/2-|\gamma_1|} \boldsymbol{\nu_1}, \dots,  s^{-1/2-|\gamma_k|} \boldsymbol{\nu_k}, s^{-\eta}\boldsymbol{\phi_{1,0}}).
\end{equation}
For all $( ( \boldsymbol{\nu_i} )_i,\boldsymbol{\phi_{1,0}}) 
f \in \m B$, we consider the solution $w(s_0, ( \boldsymbol{\nu_i} )_i,\boldsymbol{\phi_{1,0}},y,s)$ (or $w(y,s)$ for short) to  equation \eqref{eq:nlw_w}, with initial condition at time $s_0$ given by \eqref{w0} with
\[
{}^t\big((\nu_{i,0})_i,\phi_{1,0}\big) = \Gamma_{s_0}({}^t( (\boldsymbol{ \nu_i} )_i,\boldsymbol{\phi_{1,0}})).
\]
As we showed after the statement of Proposition \ref{lemode0}, $w(y,s)$ can be modulated (up to some time $\bar s = \bar s(s_0,  (\boldsymbol{ \nu_i} )_i,\boldsymbol{\phi_{1,0}})>s_0$) into a triplet $(q(s), (d_i(s))_i,(\nu_i(s))_i)$. 
From the uniqueness of such a decomposition (which is a consequence of the application of the implicit function theorem, see the proof of \cite[Proposition 3.1]{MZisol10}), we get
\begin{equation}\label{initmod}
q(s_0) =0, \quad d_i(s_0)=\hat{ d_i}(s_0,\phi_{1,0})\quad \text{and} \quad
\nu_i(s_0) = \Big[ \Gamma_{s_0}({}^t( ( \boldsymbol{\nu_i} )_i,\boldsymbol{\phi_{1,0}}))\Big]_i,\quad \textrm{for}\ \ i=1,\dots,k,
\end{equation}
where $\hat{ d_i}(s_0,\phi_{1,0})$ is defined in \eqref{v14}.

\medskip

\noindent
From \eqref{initmod}, the definitions \eqref{defxi1} of $\xi_i(s)$ and \eqref{v14} of 
 $\hat{ d_i}(s_0,\phi_{1,0})$ , we see that 
\begin{equation}\label{v22J}
\forall i=1,\dots,k, \quad  \xi_i(s_0)=\phi_{1,0}.
\end{equation}
In other words, $\big(\xi_i(s_0)\big)_i$ is in the kernel of the matrix $M$
defined in \eqref{defM} (see \eqref{385}).  In particular
\begin{equation}\label{initmod2}
\forall i=2,\dots,k, \quad  \phi_i(s_0)=0 \qquad \text{and} \quad \phi_1(s_0)=\phi_{1,0}.
\end{equation}

\medskip

Performing the change of variables \eqref{defxi1} and \eqref{defphi}, we reduce the control of $w(s)$ to the control of $(q(s), (\nu_i(s))_i, (\phi_i(s))_i)$.
Defining
\begin{equation}\label{defN}
N(( \boldsymbol{\nu_i} )_i,\boldsymbol{\phi_{1,0}}, s) := \max \left\{ s^{1/2+\eta} \| q(s) \|_{\q H} ,\ \sup_{i\ge 1} s^{1/2+|\gamma_i|} |\nu_i(s)|,\ \sup_{i \ge 1} s^\eta |\phi_i(s)| \right\},
\end{equation}
we see that $\q V(s)$ (Definition introduced in  \ref{defVa}) is simply the unit ball of the norm $N(( \boldsymbol{\nu_i} )_i,\boldsymbol{ \phi_{1,0}},s)$.

As asserted just before  \eqref{aim}, we aim at finding 
$\big((\boldsymbol{\nu_{i}})_i,\boldsymbol{\phi_{1,0}}\big)$
 so that the associated solution of equation 
\eqref{eq:nlw_w}
 $w \in \q C([s_0,\infty), \q H)$ is globally defined for forward times and
\[ \forall s \ge s_0, \quad N(( \boldsymbol{\nu_i} )_i,\boldsymbol{\phi_{1,0}},s) \le 1, \quad \text{i.e.} \quad w(s) \in \q V(s). \]
We argue by contradiction.  Assume that the conclusion of Proposition \ref{reducw} does not hold.
In particular, for all
$\big((\boldsymbol{\nu_{i}})_i,\boldsymbol{\phi_{1,0}}\big),$
the exit time $s^*\big(s_0,(\boldsymbol{\nu_{i}})_i,\boldsymbol{\phi_{1,0}}\big)$ is finite, where
\begin{equation} \label{exit}
s^*\big(s_0,(\boldsymbol{\nu_{i}})_i,\boldsymbol{\phi_{1,0}}\big)= \sup \{ s \ge s_0 \;|\; \forall \tau \in [s_0,s], \ N(( \boldsymbol{\nu_i} )_i,\boldsymbol{\phi_{1,0}},\tau) \le 1 \}.
\end{equation}
Then,  by continuity, notice that
\begin{equation} \label{Nbord}
N\big((\boldsymbol{ \nu_i} )_i,\boldsymbol{\phi_{1,0}},
s^*(s_0,(\boldsymbol{\nu_{i}})_i,\boldsymbol{\phi_{1,0}})\big) =1,
\end{equation}
and that the supremum defining $
s^*\big(s_0,(\boldsymbol{\nu_{i}})_i,\boldsymbol{\phi_{1,0}}\big)$
 is in fact a maximum.\\
We now consider the (rescaled) flow for the $(\boldsymbol{\nu_i})_i$, that is
\begin{equation}\label{defPhi}
\Phi : (s,( \boldsymbol{\nu_i}  )_i ,\boldsymbol{ \phi_{1,0}}) \mapsto \Gamma_s^{-1}({}^t(\nu_1(s), \dots , \nu_k(s),\phi_1(s))).
\end{equation}
By the properties of the flow, $\Phi$ is a continuous function of $(s,(\boldsymbol{ \nu_i} )_i, { \boldsymbol{\phi_{1,0}}}) \in [s_0,  s^*\big(s_0,(\boldsymbol{\nu_{i}})_i,\boldsymbol{\phi_{1,0}}\big)] \times \m B$.
By definition of the exit time $s^*\big(s_0,(\boldsymbol{\nu_{i}})_i,\boldsymbol{\phi_{1,0}}\big)$, we have that for all $s \in [s_0, 
s^*\big(s_0,(\boldsymbol{\nu_{i}})_i,\boldsymbol{\phi_{1,0}}\big)]$, $\Phi(s,  ( \boldsymbol{\nu_i} )_i , \boldsymbol{\phi_{1,0}}) \in \m B$.
The following claim allows us to conclude:

\medskip


\begin{cli}\label{cli27} 
For $s_0$ large enough, we have:\\
(i) For all $
\big((\boldsymbol{\nu_{i}})_i,\boldsymbol{\phi_{1,0}}\big)\in   \m B$, $\Phi \big(s^*(s_0,(\boldsymbol{\nu_{i}})_i,\boldsymbol{\phi_{1,0}}), 
(\boldsymbol{\nu_{i}})_i,\boldsymbol{\phi_{1,0}}\big)
 \in \m S$.\\
(ii) The flow $s \mapsto \Phi(s,  (\boldsymbol{ \nu_i}  )_i ,\boldsymbol{ \phi_{1,0}})$ is transverse (outgoing) at $
 s=s^*(s_0,(\boldsymbol{\nu_{i}})_i,\boldsymbol{\phi_{1,0}}\big) $ (when it hits $\m S$).\\
(iii) If $
 ((\boldsymbol{\nu_{i}})_i,\boldsymbol{\phi_{1,0}})\in \m S$, then $s^*(s_0, (\boldsymbol{\nu_{i}})_i, \boldsymbol{\phi_{1,0}})=s_0$ and $\Phi(s^*(s_0,(\boldsymbol{\nu_{i}})_i, \boldsymbol{\phi_{1,0}}),  (\boldsymbol{\nu_{i}})_i,\boldsymbol{\phi_{1,0}}) =((\boldsymbol{\nu_{i}})_i, \boldsymbol{\phi_{1,0}})$.
\end{cli}


\begin{proof}[Proof of  Claim 2.7.] In the following, the constant $C$ stands for $C(s_0)$.\\
(i) Since for all $s\in [s_0,s^*(s_0, (\boldsymbol{\nu_{i}})_i,\boldsymbol{ \phi_{1,0}})]$, $N( (\boldsymbol{\nu_{i}})_i, \boldsymbol{\phi_{1,0}},s_0)
\le 1$, it follows that, for all $i=1,\dots,k$,  we have   $| \phi_i(s)|\le C$, hence from  the change of variables \eqref{defxi1} and \eqref{defphi} together with the definition \eqref{solpart} of $\bar \zeta_i(s)$, we see that,  for all $i=1,\dots,k$,  we have
\[
|\xi_i(s)| =\frac 2{p-1}|\zeta_i(s) -\bar \zeta_i(s)-\zeta_0|\le C\]
  so that  for all $i=2,\dots,k$,  we obtain 
 \[|\zeta_i(s) - \zeta_{i-1}(s) - \frac{p-1}{2} \log s| \le C.
\]
This in turns implies that
$1/(Cs^{|\gamma_i|}) \le 1 -d_i^2 \le C/s^{|\gamma_i|}$, except for $i = (k+1)/2$ if $k$ is odd, where $1-d_i(s)^2\ge \frac 1C$.
This leads also to the bounds
\[ J \le \frac{C}{s}, \quad \bar J \le \frac{C}{s^{1/2}}, \quad 
\sum_{j=1}^k \phi_j^2
 \le \frac{C}{s^{2\eta}},
\]
where the different quantities are defined in \eqref{defJ} and  \eqref{defphi}.\\
Hence, the estimates \eqref{est:q}, \eqref{initmod}, \eqref{est:nu} and \eqref{est:phi} read as follows: for all $s \in [s_0,
s^*(s_0, (\boldsymbol{\nu_{i}})_i, \boldsymbol{\phi_{1,0}})]$
\begin{align}
\| q (s) \|_{\q H} & \le \frac{C}{s^{\bar p/2}}+\frac{C}{s^{\alpha/2}}\le \frac{1}{2 s^{1/2+ \eta}}, \text{ and from this} \label{est:q2} \\
| \dot \nu_i - \nu_i | & \le C \left( \frac{1}{s^{|\gamma_i| + \bar p}}  + \frac{1}{s^{|\gamma_i| + \alpha}}+ \frac{1}{s^{|\gamma_i| + 1}}  \right) \le \frac{C}{s^{|\gamma_i| + 1}} \label{est:nu2} \\
\left| \dot \phi_i + \frac{m_i}s \phi_i \right| & \le C \left(  \frac{1}{s^{1+2\eta}} +\frac{1}{s^{3/2}}+\frac 1{s^{(\bar p+1)/2}}+ \frac{1}{s^{1+\delta}}+ \frac{1}{s^{\alpha}} \right) \le \frac{C}{s^{1+2\eta}}, \label{est:phi2}
\end{align}
provided that $s_0$ is large enough
(note that we used the definition \eqref{defeta} of $\eta$ in the first and last line above).
Now, if $i =2, \dots, k$, recall from the definition \eqref{defeta} of $\eta$  that $0<2\eta<m_i$.
Considering $g_i(s) = s^{m_i} \phi_i(s)$, we see that $|\dot g_i(s)| \le C s^{m_i-(1+2\eta)}$. Since $\phi_i(s_0)=0$ by \eqref{initmod2}, we write
\begin{equation} \label{phibound}
| \phi_i(s)| \le   \frac{C}{s^{2\eta}} \le \frac{1}{2 s^\eta},
\end{equation}
for $s_0$ large enough.

\medskip

Since
$N\big(( \boldsymbol{\nu_i} )_i,\boldsymbol{\phi_{1,0}},
s^*(s_0,(\boldsymbol{\nu_{i}})_i,\boldsymbol{\phi_{1,0}})\big) =1$
by \eqref{Nbord}, we see from the definition \eqref{defN} of $N$ together with \eqref{est:q2} and \eqref{phibound} that necessarily there exists $i=1,\dots,k$ such that
\begin{equation*}
s^*(s_0,(\boldsymbol{\nu_{i}})_i,\boldsymbol{\phi_{1,0}})^{1/2+|\gamma_i|}|\nu_i(s^*(s_0,(\boldsymbol{\nu_{i}})_i,\boldsymbol{\phi_{1,0}}))|=1,\ \textrm{or}\   s^*(s_0,(\boldsymbol{\nu_{i}})_i,\boldsymbol{\phi_{1,0}})^{\eta}|\phi_1(s^*(s_0,(\boldsymbol{\nu_{i}})_i,\boldsymbol{\phi_{1,0}}))|=1.
\end{equation*}
Using the definitions \eqref{defPhi} and \eqref{rescl} of the flow $\Phi$ and the rescaling function $\Gamma_s$, we get to the conclusion of part (i). 

\medskip

\noindent (ii) Assume that $\Phi(s,(\boldsymbol{\nu_{i}})_i,\boldsymbol{\phi_{1,0}}) \in \m S$ for some $s\in [s_0,s^*(s_0, (\boldsymbol{\nu_{i}})_i,\boldsymbol{\phi_{1,0}})]$. Therefore, 
either there exists $i=1,\dots,k$ such that
\begin{equation}\label{iibord}
s^{1/2+|\gamma_i|}|\nu_i(s)|=1, 
\end{equation}
or 
\begin{equation}\label{ibord}
\  s^{\eta}|\phi_1(s)|=1.
\end{equation}
In the case where \eqref{iibord}  holds, by  using \eqref{est:nu2}, we write
\begin{align*}
\MoveEqLeft
\frac{d}{ds} s^{1/2+|\gamma_i|} \nu_i(s)
 = {s}^{1/2 + |\gamma_i|} \left( \big( \frac{1}{2} + | \gamma_i|\big) \frac{\nu_i(s)}{s} + \dot \nu_i(s) \right) \\
& =  {s}^{1/2 + |\gamma_i|} \left(  \nu_i(s) \big( 1 + \frac{1}{2s} + \frac{|\gamma_i|}{s} \big) + O \big( \frac{1}{{s}^{1+|\gamma_i|}} \big) \right) \\
& =  {s}^{1/2 + |\gamma_i|} \left(  \nu_i(s) + O \left( \frac{1}{{s}^{1+|\gamma_i|}} \right) \right).
\end{align*}
By  \eqref{iibord}, we deduce that for $s_0$ large enough,
\begin{equation*} 
\frac{d}{ds} s^{1/2+|\gamma_i|} \nu_i(s)\cdot \frac{1}{ {s}^{1/2 +|\gamma_i|} \nu_i(s) } \ge \frac{1}{2}.
\end{equation*}
Now, in the case where  \eqref{ibord} holds,
 using the fact that $m_1=0$ (see \eqref{defmi}),  we derive from the inequality  \eqref{est:phi2} the fact that
$$|\frac{d}{ds} \phi_1(s^*)|
 \le \frac{\eta}{{2(s^*)}^{1+\eta}}.$$ Since
$\frac{d}{ds}(\frac1{s^{\eta}})=
 -\frac{\eta}{{s}^{1+\eta}}$, then we have
\begin{itemize}
 \item  $\frac{d}{ds} \phi_1(s^*) > \frac{d}{ds} \frac1{s^{\eta}}|_{s=s^*}$, if $ \phi_1(s^*)= \frac1{{s^*}^{\eta}}$,
 \item  $\frac{d}{ds} \phi_1(s^*) < \frac{d}{ds} \frac{-1}{s^{\eta}}|_{s=s^*}$, if $ \phi_1(s^*)= -\frac1{{s^*}^{\eta}}$,
\end{itemize}
 Thus, the flow is transverse on $\m B$ and part (ii) holds.  Note that the transversality  of  $ \phi_1$ is a new feature in our approach, which was not present in 
 C\^ote and Zaag \cite{CZcpam13}.

\medskip

%

\noindent (iii) Let  $((\boldsymbol{\nu_{i}})_i,\boldsymbol{\phi_{1,0}})\in \m S$. From \eqref{initmod} and the definition \eqref{defPhi} of the flow $\Phi$, we see that
\begin{equation}\label{ini}
\Phi(s_0, (\boldsymbol{\nu_{i}})_i, \boldsymbol{\phi_{1,0}}) = ( (\boldsymbol{\nu_{i}})_i, \boldsymbol{\phi_{1,0}}).
\end{equation}
 Since $((\boldsymbol{\nu_{i}})_i,\boldsymbol{\phi_{1,0}})\in \m S$, we can use (ii) of  Claim 2.7 and see that the flow $\Phi$ is transverse to $\m B$ at $s=s_0$. By definition of the exit time, we see that
\[s^*(s_0, (\boldsymbol{\nu_{i}})_i,\boldsymbol{\phi_{1,0}}) = s_0.
\]
Using \eqref{ini}, we get to the conclusion of part (iii). This concludes the proof of   Claim 2.7.
\end{proof}
We now conclude the proof of Proposition \ref{reducw}. From  part (ii) of Claim 2.7, $
((\boldsymbol{\nu_{i}})_i,\boldsymbol{\phi_{1,0}}) \to s^*(s_0, (\boldsymbol{\nu_{i}})_i,\boldsymbol{\phi_{1,0}})$ is continuous, hence from (i) and (iii),
\[
((\boldsymbol{\nu_{i}})_i, \boldsymbol{\phi_{1,0}})
 \mapsto \Phi(s^*(s_0,(\boldsymbol{\nu_{i}})_i,\boldsymbol{ \phi_{1,0}}),
(\boldsymbol{\nu_{i}})_i,\boldsymbol{ \phi_{1,0}})\]
is a continuous map from $\m B$ to $\m S$ whose restriction to $\m S$ is the identity. By the  index theory, this is  a contradiction. Thus, there exists $((\boldsymbol{\nu_{i}})_i, \boldsymbol{\phi_{1,0}}) \in \m B$ such that for all $s\ge s_0$, $N(s_0,
((\boldsymbol{\nu_{i}})_i, \boldsymbol{\phi_{1,0}})\le 1$, hence $w(s_0, ((\boldsymbol{\nu_{i}})_i, \boldsymbol{\phi_{0,1}}, \cdot,s) \in \q V( s)$. This is the desired conclusion. This concludes the proof of  Proposition \ref{reducw}.
\end{proof}

It remains to give the proof of Proposition \ref{propw} in order to conclude this section. Let us first recall from Lemma A.2 in \cite{MZisol10} the following continuity result for the family of solitons $\kappa^*(d,\nu)$:
\begin{lem}[Continuity of $\kappa^*$] \label{contk*}For all $A\ge 2$, there exists $C(A)>0$ such that if $(d_1,\nu_1)$ and $(d_2, \nu_2)$ satisfy
\begin{equation}\label{condA}
\frac {\nu_1}{1-|d_1|},\frac {\nu_2}{1-{|d_2|}}\in [-1+\frac 1A, A],\;\;
\end{equation}
then
\begin{multline}\label{defze}
\|\kappa^*(d_1,\nu_1)-\kappa^*(d_2, \nu_2)\|_{\H} \\
\le C(A)\left(\left| \frac {\nu_1}{1-|d_1|} -\frac {\nu_2}{1-{|d_2|}}\right| +\left|\arg\tanh d_1 - \arg\tanh d_2\right|\right).
\end{multline}
\end{lem}
With this lemma, we can give the proof of Proposition \ref{propw}.

\begin{proof}[Proof of Proposition \ref{propw}]
Let us consider the solution constructed in Proposition \ref{reducw}. Since $w (s) \in \q V(s)$ for all $s \ge s_0$, from Corollary \ref{cordynphi} and the definition \ref{defVa} of $\q V(s)$, we see that \eqref{est:phi2} holds. In particular, for $i=1,\dots,k$, we see that
\[
|\phi_i(s)|\le \frac C{s^{\eta}}, \qquad  \textrm{for} \qquad  s\qquad \textrm{large enough}.
\]
Therefore, for $i=1,\dots,k$,  we have $\phi_i(s)\to 0$  as $s\to \infty$. Then, by \eqref{defphi},   we see that $\xi_i(s) \to 0$, for $i=1,\dots,k$. From \eqref{defxi1}, we see that $\zeta_i(s) -\bar \zeta_i(s)-\zeta_0\to 0$ for all $i=1,\dots,k$ and \eqref{equid} follows. In particular,
\[
1-|d_i(s)|\sim C_i s^{-|\gamma_i|}\mbox{ as }s\to \infty,
\]
hence, from the definition \ref{defVa} of $\q V(s)$, we have
\[
\forall s\ge s_0,\;\; \frac{|\nu_i|}{1-|d_i(s)|}\le C(s_0)s^{-\frac 12}.
\]
Therefore, Lemma \ref{contk*} applies and since $\kappa^*(d_i(s),0,y)=\Big(\kappa(d_i(s),y),0\Big)$, we have
\[
\|\kappa^*(d_i(s),\nu_i(s))-(\kappa(d_i(s),0))\|_{\q H}\le C(s_0)\frac{|\nu_i|}{1-|d_i(s)|}\le C(s_0)s^{-\frac 12}.
\]
As $\|q(s)\|_{\q H}\le \frac C{s^{\frac 12 +\eta}}$ by definition \ref{defVa} of $\q V(s)$, 
and with \eqref{defq}, we deduce that  
\[
\left\|\vc{w(s)}{\ps w(s)} - \vc{\ds\sum_{i=1}^{k} (-1)^{i+1}\kappa({d}_i(s))}0\right\|_{\H} \le \|q(s)\|_{\q H} + C(s_0)s^{-\frac 12}\le C(s_0)s^{-\frac 12}
\]
and \eqref{cprofile} follows. This concludes the proof of Proposition \ref{propw}.
\end{proof}

\bigskip

Now, we are ready  to give the proof of Theorem 1 (as 
we said at the beginning of the section, we don't prove Corollary 2 as it follows directly from Theorem 1).

\subsection*{Proof of Theorem \ref{mainth}}

The proof is very easy, since we have only
to translate the construction of the previous section into the $u(x,t)$ setting, and recover a solution to our problem. This part contains straightforward and obvious arguments which may be skipped by specialists. We give them for the reader's convenience.\\

Consider an integer $k\ge 2$ and consider $w(y,s)$ the solution of \eqref{eq:nlw_w} constructed in Proposition \ref{propw}.\\
Then, let us define $u(x,t)$ as the solution of equation \eqref{PNLW} with initial data in $\rm H^1_{\rm loc,u}\times \rm L^2_{\rm loc,u}(\m R)$ whose trace in $(-1,1)$ is given by
\begin{equation}\label{u0u1}
u(x,0)=w(x,s_0)\mbox{ and }\partial_t u(x,0)=\partial_s w(x,s_0)+\frac 2{p-1}w(x,s_0)+x\partial_y w(x,s_0).
\end{equation}

Then $u(x,t)$ satisfies all the requirements in Theorem \ref{mainth}. Indeed, by the finite speed of propagation,  we clearly have:

\medskip

(i) {\it For all $t\in[0,1)$ and $|x|<1-t$,
\begin{equation}\label{uinside}
u(x,t)=(1-t)^{-\frac 2{p-1}}w\left(\frac x{1-t}, s_0-\log(1-t)\right).
\end{equation}
}

Indeed, by definition \eqref{def:w} of similarity variables, the function on the right-hand side of \eqref{uinside} is a solution to equation \eqref{PNLW} with the same initial data \eqref{u0u1} as $u(x,t)$. Since that initial data is in  $H^1\times L^2(-1,1)$ and equation \eqref{PNLW} is well-posed in $H^1\times L^2$ of sections of backward light cones, both solutions are equal from the uniqueness to the Cauchy problem and the finite speed of propagation, hence \eqref{uinside} holds.
In particular, from \eqref{def:w}, we have
\begin{equation}\label{egal}
\forall s\ge 0,\;\;\forall y\in(-1,1),\;\;
w_{0,1}(y,s)=w(y,s+s_0).
\end{equation}

(ii) {\it $u$ is a blow-up solution}. Indeed, if not, then $u$ is global and $u\in L^\infty_{loc}([0,\infty),\rm H^1_{loc,u}\times L^2_{loc,u}(\m R))$. In particular, we write from the Sobolev injection, for all $s\ge 0$ and $\epsilon>0$, 
\begin{equation}\label{tozero}
\|w_{0,1}(s)\|_{L^2_\rho}\le C \|u\|_{L^\infty(|x|<1+\epsilon-t)} e^{-\frac {2s}{p-1}}\to 0\mbox{ as }s\to\infty.
\end{equation}
This is in contradiction with \eqref{egal} and \eqref{cprofile}.

\medskip

(iii) {\it $T(0)=1$}. Indeed, from \eqref{egal} we see that $u(x,t)$ is defined in the cone $|x|<1-t$, $t\ge 0$,
 hence $T(0)\ge 1$. From \eqref{tozero}, we see that if $T(0)>1+\epsilon$ for some $\epsilon>0$,
 then the same contradiction follows. Thus $T(0)=1$.

\medskip
(iv)  From above, we can use the simplified notation for \eqref{def:w} and write $w_0$ instead of $w_{0,1}$, and rewrite \eqref{uinside} as follows:
\begin{equation*}
\forall s\ge 0,\;\;\forall y\in(-1,1),\;\;
w_0(y,s)=w(y,s+s_0).
\end{equation*}
Using \eqref{cprofile} and \eqref{equid},
we see that \eqref{cprofile0} follows for $w_0$ with
\[
\zeta_i(s) -\bar \zeta_i(s)\to\zeta_0\mbox{ as }s\to\infty\mbox{ for }i=1,\dots,k
\]
where $\zeta_0\in \m R$ and $(\bar \zeta_i(s))_i$   is the explicit solution of system \eqref{eq:tl}.
Thanks to the following continuity result for the solitons $\kappa(d)$ (which follows from estimate   \eqref{defze} in Lemma \ref{contk*}):
\begin{equation}\label{contkd}
\|\kappa(d_1)-\kappa(d_2)\|_{\q H_0} \le C|\arg\tanh d_1-\arg \tanh d_2|,
\end{equation}
 we see that \eqref{cprofile0} still holds if we slightly modify the $\zeta_i(s)$ by putting $\zeta_i(s)=\bar \zeta_i(s)+\zeta_0$ as required by \eqref{refequid1}.
Finally, if we consider the more restrictive hypothesis $(H')$,
from the classification of the blow-up behavior for general solutions, available from   \cite{HZBSM}, we clearly see that the origin is a characteristic point.
Thus, we have a solution obeying all the requirements of Theorem \ref{mainth}.



\appendix

\section{Dynamics of equation \eqref{eq:nlw_w} near multi-solitons}\label{appdyn}

This appendix is devoted to the proof of Proposition \ref{propdyn}. In the special case
where $(f,g)\equiv (0,0)$,  the proof is already known from  
 \cite[Proposition 3.4]{CZcpam13}. Note that this proof 
is based on Lemma C.2, Claims 4.8 and 4.9 of \cite{MZisol10}, together with Proposition 3.2 in \cite{MZajm11}.

\medskip

 In our case,
where  $(f,g)\not\equiv (0,0)$, in order to avoid unnecessary  repetition, we kindly refer the reader to \cite{CZcpam13} and  \cite{MZisol10} for all the projections of the terms in (\ref{eq:qq}) 
not involving
$f$ and $g$  and we will  only focus here on the terms with $f$ and $g$.

%
%

\begin{proof}[Proof of Proposition \ref{propdyn}]

Using the definition \eqref{defq} of $q$, we transform equation \eqref{eq:nlw_w} satisfied by   $w$ into the following system satisfied by $q$, for all $s\in [s_0, \bar s)$:
\begin{align}\label{eq:qq}
\ds\frac \partial {\partial s}
\vc{q_1}{q_2} & =\ll \vc{q_1}{q_2}
-\sum_{j=1}^\k(-1)^j\left[(\nu_j'(s)-\nu_j(s))\pnu\kappa^*+d_j'(s)\partial_d \kappa^*\right](d_j(s),\nu_j(s),y)\nonumber \\
&  +\vc{0}{R} + \vc{0}{h(q_1)} +\vc{0}{\tilde f(q_1)}+
\vc{0}{\tilde g(q_1)},
\end{align}
where
\begin{align}
\ll\vc{q_1}{q_2} & = \vc{q_2}{   \q L  q_1+\psi q_1-\frac{p+3}{p-1}q_2-2y\py q_2},\label{vd1}\\
\psi(y,s)& = p|K^*_1(y,s)|^{p-1} -\frac{2(p+1)}{(p-1)^2}, \qquad
K^*_1(y,s) = \sum_{j=1}^\k  (-1)^j\kappa^*_1(d_j(s),\nu_j(s),y),\label{vd2} \\
h(q_1) & = |K^*_1+q_1|^{p-1}(K^*_1+q_1)- |K^*_1|^{p-1}K^*_1- p|K^*_1|^{p-1} q_1, \label{vd3}\\
\tilde f(q_1)&=e^{-\frac{2ps}{p-1}}f\Big(e^{\frac{2s}{p-1}}(K^*_1+q_1)\Big)\label{vd4}\\
\tilde g(q_1)&=
e^{-\frac{2ps}{p-1}}g\Big(x_0+ye^{-s},T_0-e^{-s},e^{\frac{2s}{p-1}}w,e^{\frac{(p+1)s}{p-1}}\partial_yw,e^{\frac{(p+1)s}{p-1}}(\partial_sw+y\partial_y
w+\frac{2w}{p-1})\Big)\label{vd5}\\
R & = |K^*_1|^{p-1}K^*_1- \sum_{j=1}^\k (-1)^j\kappa^*_1(d_j(s),\nu_j(s),y)^p,\qquad \qquad
w=K^*_1+q_1.\label{vd6}
\end{align}

\medskip

As in the unperturbed case,
 we give a decomposition of the solution which is well adapted to the proof. More precisely,
 we start by   localizing  equation (\ref{eq:qq}) near the center 
$d^*_i(s)=\frac {d_i(s)}{1+\nu_i(s)}$ of $\kappa^*(d_i(s), \nu_i(s))$  for each $i=1,\dots,k$, which allows us to view it locally as a perturbation of the case of $\kappa(d,y)$ already treated in \cite{MZjfa07}. For this, 
given $i=1,\dots,k$, we need to expand the linear operator of equation 
 (\ref{eq:qq}) 
 as  follows:
\begin{equation}\label{expl}
\ll(q) = \LL_{d_i^*(s)}(q) + (0, \bar V_i(y,s) q_1)+ (0, V_i^*(y,s) q_1),
\end{equation}
with
\begin{eqnarray}
\LL_{d}\vc{q_1}{q_2}&=&\vc{q_2}{  \q L  q_1+\psi^*(d)q_1-\frac{p+3}{p-1}q_2-2yq_2'},\label{defld}\\
\psi^*(d,y)&=&p\kappa(d,y)^{p-1}-\frac{2(p+1)}{(p-1)^2},\label{defpsi*}\\
\bar V_i(y,s)&=&p \kappa_1^*(d_i(s), \nu_i(s),y)^{p-1}-p\kappa(d_i^*(s),y)^{p-1},\label{defbvi}\\
V_i^*(y,s)&=& p|K^*_1(y,s)|^{p-1} - p \kappa^*_1(d_i(s), \nu_i(s),y)^{p-1}.\label{defvi-09}
\end{eqnarray}
We proceed in two parts.

- In Part 1, in order to prove \eqref{est:nu} and \eqref{est:zeta}, we project equation \eqref{eq:qq} using the projector $\pp_l^{d_i^*(s)} $ defined in \eqref{defpdi} with $l=0,1$ and $d_i^*(s) = \frac{d_i(s)}{1+\nu_i(s)}.$

- In Part 2, we will find a Lyapunov functional for equation \eqref{eq:qq}, which is equivalent to the norm squared, and we thus 
 obtain estimate \eqref{est:q}.

\bigskip

{\bf Part 1: Projection of equation \eqref{eq:qq} on   $F_0^{d_i^*(s)}$ and $F_1^{d_i^*(s)}$}

Let us assume that $s_0\ge 1$ is  large enough  and take $s\in [s_0, \bar s)$.
\medskip

The proof of  (\ref{est:nu}) and   (\ref{est:zeta})  is similar to the proof in \cite{CZcpam13}
except for the treatment of the perturbation terms.  More precisely,
 let $i=1,\dots,k$ be fixed  and $l=0$ or $1$,
the projector 
 $\pp_l^{d_i^*(s)} $ 
defined in (\ref{defpdi}) is now applied for 
 each term of equation (\ref{eq:qq}).
 Thanks to (\ref{conmod}) together with  the analysis of  
\cite[Appendix C]{MZajm11},  \cite[Appendix C]{MZisol10} and \cite[Appendix A]{CZcpam13}, we easily  obtain the following estimates related to the
terms not  involving $f$ and $g$:
\begin{eqnarray}\label{s0}
\begin{array}{l}
|\Pi^{d_i^*}_l(\partial_s q)|\le 
C\displaystyle{
\frac{|d_i'|+|\nu_i'-\nu_i|+|\nu_i| }{1-{d_i^*}^2}\|q\|_{\H}},
\\
\Pi^{d_i^*}_l(\LL_{d_i^*}(q))=l\Pi^{d_i^*}_l\left(q\right)=0,\\
\\
|\Pi^{d_i^*}_l(0,h(q_1))|\le C \|q\|_{\H}^2 ,\\
\\
|\Pi^{d_i^*}_l(0,V_i^*q_1)| \le C \|q\|_{\H}^2 
+C {(J^*)}^{1+\delta_1},\\
\\
|\Pi^{d_i^*}_l
(0,{\bar V_i} q_1)|
 \le C  \displaystyle{ \frac{|\nu_i|}{1-{d_i^*}^2}\|q\|_{\H}},\\
|\Pi^{d_i^*}_l(0,R)|\le \displaystyle{CJ^*},\\
|\Pi_0^{d_i^*}(0,R)-c_2(p)(-1)^i[ e^{-\frac 2{p-1}(\zeta_i-\zeta_{i-1})}-e^{-\frac 2{p-1}(\zeta_{i+1}-\zeta_i)}] |\le CJ^{1+\delta_2}+CJ\bar J,
\end{array}
\end{eqnarray}
where  $\delta_1 >0$, $\delta_2 >0$ and $c_2(p)>0$, where 
\begin{equation}\label{Jstart}
J^*=\displaystyle{\sum_{i=1 }^{\k-1 } e^{-\frac 2{p-1}(\zeta^*_{i+1}-\zeta^*_i) }},
\end{equation}
 and $\zeta^*_{i}=- \argth (d_i^*) $. 
Moreover,  we recall  some estimates related to  the projections of  $\partial_{\nu} \kappa^*(d_i,\nu_i)$
and $\partial_{d} \kappa^*(d_i,\nu_i)$  given in \cite[Claim 3.2]{MZisol10} and  in  \cite[Appendix A]{CZcpam13}
\begin{align}\label{s0000}
&\Pi_0^{d_i^*}(\partial_{\nu} \kappa^*(d_i,\nu_i))=0,\nonumber\\
&-\frac {C}{1-{d_i^*}^2}\le \Pi_1^{d_i^*}(\partial_{\nu} \kappa^*(d_i,\nu_i))\le -\frac 1{C(1-{d_i^*}^2)},\nonumber\\
&|\Pi_1^{d_i^*}(\partial_d \kappa^*(d_i,\nu_i))|\le \frac{C}{1-{d_i^*}^2},
\\
&-\frac {C}{1-{d_i^*}^2}\le \Pi_0^{d_i^*}(\partial_d \kappa^*(d_i,\nu_i))\le -\frac 1{C(1-{d_i^*}^2)},\nonumber\\
&\left|\Pi_l^{d_i^*}(\partial_{\nu} \kappa^*(d_j,\nu_j))\right|+\left|\Pi_l^{d_i^*}(\partial_d \kappa^*(d_j,\nu_j))\right|\le \frac{C}{1-{d^*_j}^2}\sqrt{J^*}, \quad \textrm{for}\quad  i\neq j,\nonumber\\
&\left|\pp_0^{{d_i^*}}(\partial_d \kappa^*(d_i,\nu_i))+\frac {c_3}{1-d_i^2}\right|\le C\bar J,\nonumber
\end{align}
for some $c_3>0$.

\medskip

In the sequel, we shall intensively  use  the following inequalities
\begin{equation}\label{equiv11}
C^{-1}J\le J^*\le CJ,\qquad    C^{-1}\le \frac{1-d_i^2}{1-d_i^{*2}}\le C,
\end{equation}
and
\begin{equation}\label{equiv112}
C^{-1}(J^*)^{\frac{\bar p}2}\le \hat J^*\le C(J^*)^{\frac{\bar p}2},
\end{equation}
where $J$  and $J^*$ are   defined in \eqref{defJ}  and \eqref{Jstart} and 
\begin{equation}\label{equiv114}\displaystyle{ \hat J^* = \sum_{i=1}^{k-1} e^{- \frac{\bar p}{p-1}(\zeta^*_{i+1} - \zeta^*_{i})}}.
\end{equation}
As for the proof of  \eqref{equiv11} 
and \eqref{equiv112}, note that   \eqref{equiv11} easily follows from
  (\ref{conmod}), and that 
 \eqref{equiv112}  follows from the fact that all the norms on $\R^{k-1}$
 are equivalent.

%

\medskip

Now, we focus on the ``new'' terms of the problem related to the  perturbation terms $f$ and $g$.



\medskip

Let us first recall from 
our previous papers the following  basic bounds  on the solitons $\kappa (d,y)$ defined in \eqref{defkd} and $\kappa^*(d,\nu,y)$ 
defined in  (\ref{defk*}-\ref{defk**}) together with a crucial Hardy Sobolev  inequality:
\begin{lem}{\bf{(Useful estimates)}}\label{use}{\ }\\
(i){\bf{(A Hardy-Sobolev type inequality)}} For all $h\in \H_0$, we have
\begin{equation*}
\|h\|_{L^2_{\frac \rho{1-y^2}(-1,1)}}+\|h\|_{L^{p+1}_\rho(-1,1)}+ \|h(1-y^2)^{\frac 1{p-1}}\|_{L^\infty(-1,1)}
\le C\|h\|_{\H_0}.
\end{equation*}
(ii){\bf{ (Boundedness of $\kappa (d, y)$  in several norms)}}  For all  $d\in (-1,1)$, we have
\begin{equation*}
\|\kappa(d,y)\|_{L^2_{\frac \rho{1-y^2}(-1,1)}}+\|\kappa(d,y)\|_{L^{p+1}_\rho(-1,1)}+ \|\kappa(d,y)(1-y^2)^{\frac 1{p-1}}\|_{L^\infty(-1,1)}
+\|\kappa(d,y)\|_{\H_0}\le C.
\end{equation*}
(iii) For $l=0,1$ and any $|y|<1$, $|W_{l,2}( d,y)|\le C \kappa( d,y)\le C(1-y^2)^{-\frac 1{p-1}}$.\\ 
(iv)
{\bf{ (Properties of $\kappa^*(d,\nu,y)$ )}} For all $d\in (-1,1)$ and $\nu >-1+|d|$, we have:
\begin{eqnarray*}
\forall y\in(-1,1),\;\;0\le \kappa_1^*(d,\nu,y)=\lambda \kappa\left(\frac d{1+\nu},y\right)&\le & \frac{C \lambda}{(1-y^2)^{\frac 1{p-1}}},\\
\left\|\kappa^*\left(d,\nu\right)\right\|_{\H} &\le& C\lambda +C1_{\{\nu<0\}}\frac{|\nu|}{\sqrt{1-d^2}}\lambda^{\frac{p+1}2},
\end{eqnarray*}
where $\lambda$ is defined  by
\begin{equation*}
\lambda= \lambda(d,\nu)=\frac{(1-d^2)^{\frac 1{p-1}}}{[(1+\nu)^2-d^2]^{\frac 1{p-1}}}.
\end{equation*}
(v)  For all $d\in (-1,1)$ and $\nu >-1+|d|$, we have:
  $$\forall y\in(-1,1),  \ \ |\partial_{d} \kappa^*_1(d, \nu, y)|+
 |\partial_{\nu} \kappa^*_1(d, \nu,y)|\le C\ds{  \frac{ \kappa (d^*,y)}{1-d^{*2}}},  \ \textrm{where}\ \  d^* = \frac{d}{1+\nu}.$$
\end{lem}

\begin{proof}- For  (i), see Lemma 2.2 page 51 in  \cite{MZjfa07}.
For (ii), use (i) and  identity (49) page 59 in \cite{MZjfa07}. For (iii), see (196)  page 105 in  \cite{MZjfa07}.
For (iv),  see (i) in Lemma A.2 page 2878  in   \cite{MZisol10}.   For (v), see (3.41) page 2857  in    \cite{MZisol10}.
\end{proof}

\bigskip

In  the following, we give the estimates  involving the   perturbation terms. 
\begin{itemize}
\item Estimate of  $\Pi^{d_i^*}_l({0},{\tilde f(q_1))}$, for $l=0,1$.

\medskip

Using the definition \eqref{defpdi} and  (iii) in Lemma \ref{use}, we  conclude that
\begin{equation}
|\Pi^{d_i^*}_l({0},{\tilde f(q_1))}|\le   C\iint  |\tilde f(q_1)|\kappa(d^*_i )\rho dy. \label{t0}
\end{equation}
By using the definition \eqref{vd4} of $\tilde f$,  the condition  $(H)$ and 
(iii) in Lemma \ref{use},  we deduce that
\begin{equation}
 |\Pi^{d_i^*}_l({0},{\tilde f(q_1))}| \leq Ce^{-\frac{2ps}{p-1}} 
+C \iint\frac{|K^*_1+q_1|^{p} }{(\log(2+e^{\frac{4s}{p-1}}(K^*_1+q_1)^2))^{\alpha}}\kappa(d^*_i)\rho dy.\label{tv1}
\end{equation}
In order to bound  this latter integral, we divide the interval $[-1,1]$ in two parts:
 $$A_{1}(s)=\{y \in (-1,1) \,\,|\,\,|K^*_1+q_1|\leq  e^{-\frac{s}{p-1}}\}\,\,{\rm and }\,\,A_{2}(s)=\{y \in  (-1,1)\,\,|\,\, |K^*_1+q_1|>  e^{-\frac{s}{p-1}}\}.$$
On the one hand, ${\rm if }\,\,y \in A_{1}(s)$, we have
\begin{align}\label{nov3}
\frac{|K^*_1+q_1|^{p} }{(\log(2+e^{\frac{4s}{p-1}}(K^*_1+q_1)^2))^{\alpha}}\le  C e^{-\frac{ps}{p-1}}.
\end{align}
Integrating  \eqref{nov3}  over $A_{1}(s)$ and using
 (iii) in Lemma \ref{use},  we see that
\begin{equation}
\int_{A_{1}(s)}\frac{|K^*_1+q_1|^{p} }{(\log(2+e^{\frac{4s}{p-1}}(K^*_1+q_1)^2))^{\alpha}}\kappa(d^*_i)\rho dy
\le C e^{-\frac{ps}{p-1}}. \label{t1}
\end{equation}
On the other hand, if $y\in A_{2}(s)$, we have
$$\log(2+e^{\frac{4s}{p-1}}(K^*_1+q_1)^2)>\log(2+e^{\frac{2s}{p-1}}) \geq \frac{2s}{p-1},  $$
and we write,  for all$\,\,y \in A_{2}(s)$,
\begin{align}
\frac{|K^*_1+q_1|^{p} }{(\log(2+e^{\frac{4s}{p-1}}(K^*_1+q_1)^2))^{\alpha}}\le   \frac{C}{s^{\alpha}} |K^*_1+q_1|^{p}\le   \frac{C}{s^{\alpha}}\Big( |K^*_1|^{p}+  |q_1|^{p}\Big).\label{tv2}
\end{align}
If we integrate  \eqref{tv2} over $A_{2}(s)$, using the simple fact that $A_{2}(s)\subset [-1,1],$ we obtain
\begin{equation}
\int_{A_{2}(s)}\!\! \frac{|K^*_1+q_1|^{p} }{(\log(2+e^{\frac{4s}{p-1}}(K^*_1+q_1)^2))^{\alpha}}\kappa(d^*_i)\rho dy
\le  \frac{C}{s^{\alpha}}\iint\!  \big(|K^*_1|^{p}+
|q_1|^{p}\big)\kappa(d^*_i)\rho dy.
\label{t2}
\end{equation}
Furthermore, we apply  the  inequality  $ab^p\le C( a^{p+1}+b^{p+1})$,  for all $a>0, b>0$,  to get
\begin{align}
 \iint |q_1|^{p}\kappa(d^*_i)\rho dy\le   C\iint |q_1|^{p+1}\rho dy+C\iint \kappa^{p+1}(d^*_i)\rho dy.
\label{t22}
\end{align}
Gathering \eqref{t2},  \eqref{t22},    Lemma \ref{use}   and (\ref{conmod}), we have
\begin{align}
 \iint |q_1|^{p}\kappa(d^*_i)\rho dy\le   C\|q\|_{\H}^{p+1}+C \le C.\label{t222}
\end{align}
 It suffices to combine    \eqref{tv1}, \eqref{t1}, \eqref{t2},  \eqref{t222}  and (iv) in   Lemma \ref{use},   to obtain that
\begin{align}
|\Pi^{d_i^*}_l({0},{\tilde f(q_1))}|\le \displaystyle{\frac{C}{s^{\alpha}}}.\label{s1}
\end{align}

\item Estimate of  $\Pi^{d_i^*}_l({0},{\tilde g(q_1))}$,  for $l=0,1$.

\medskip

 Proceeding as for  estimate \eqref{t0},
 we  have 
\begin{equation}
|\Pi^{d_i^*}_l({0},{\tilde g(q_1))}|\le   C\iint 
|\tilde g(q_1)| \kappa(d^*_i ) \rho dy. \label{g0}
\end{equation}
By using the condition  $(H)$ and the definition \eqref{vd5} of $\tilde g$, we obtain 
\begin{align}\label{nov8}
 |{\tilde g(q_1)}| \le &
Ce^{-s}  (|\partial_yK^*_1+\partial_yq_1|+|\partial_sK^*_1+\partial_sq_1|)\nonumber\\
&+Ce^{-s}|K^*_1+q_1|+|K^*_1+q_1|^{\frac{p+1}2}) 
+Ce^{-\frac{2ps}{p-1}},
\end{align}
where $K^*_1$ is defined in \eqref{vd2}.
Combining    \eqref{nov8} and    the  expression 
\begin{equation}
q_2=\partial_sq_1+\partial_sK^*_1-\sum_{j=1}^\k  (-1)^j \nu_j 
\partial_{\nu} \kappa^*_1(d_j, \nu_j) \label{q21}
\end{equation}
(which comes from  \eqref{defq}),  one easily obtains
\begin{align}
 |{\tilde g(q_1)}| 
 \leq &
Ce^{-s}  \big(|\partial_yq_1|+|q_2|+|q_1|+|q_1|^{\frac{p+1}2}\big)
+
Ce^{-s}  \big(|\partial_yK^*_1|+|K^*_1|+|K^*_1|^{\frac{p+1}2}\big)\nonumber\\
& +Ce^{-s} 
\sum_{j=1}^\k  | \nu_j| 
|\partial_{\nu} \kappa^*_1(d_j, \nu_j)| +
Ce^{-\frac{2ps}{p-1}}.\label{8novbis}
\end{align}
According to  \eqref{g0},  \eqref{8novbis},  together with items (i)  and
 (ii) in   Lemma \ref{use}, we deduce that
\begin{align}
|\Pi^{d_i^*}_l({0},{\tilde g(q_1))}|\le \ \ \ \ \ &\underbrace{   Ce^{-s}\iint 
  \big(|\partial_yq_1|+|q_2|+|q_1|+|q_1|^{\frac{p+1}2}\big)\kappa(d^*_i ) \rho dy}_{I_1(s)}\nonumber\\
&+\underbrace{
Ce^{-s}\iint  \big(|\partial_yK^*_1|+|K^*_1|+|K^*_1|^{\frac{p+1}2}\big)\kappa(d^*_i ) \rho dy}_{I_2(s)}\label{8nov2}\\
&+\underbrace{Ce^{-s} 
\sum_{j=1}^\k  | \nu_j| \iint
|\partial_{\nu} \kappa^*_1(d_j, \nu_j)| \kappa(d^*_i ) \rho dy}_{I_3(s)}+
Ce^{-\frac{2ps}{p-1}}.\nonumber
\end{align}
We are going now to estimate the different terms of the right-hand  side of inequality  \eqref{8nov2}. Thanks to the the classical
inequality $ab\le a^2+b^2$,   (ii) in   Lemma \ref{use} and 
 (\ref{conmod}), we conclude that 
\begin{align}\label{8nov3}
|I_1(s)|\le  Ce^{-s}\iint 
  \Big(|\partial_yq_1|^2(1-y^2)+|q_2|^2+|q_1|^2+|q_1|^{p+1}+\frac{\kappa^2(d^*_i )}{1-y^2}\Big) \rho dy\le Ce^{-s}.
\end{align}
 Proceeding as for  estimate \eqref{8nov3} and exploiting,    (\ref{conmod}),   (i) and (iv) in   Lemma \ref{use},  we conclude
\begin{align}\label{8nov4}
|I_2(s)|\le Ce^{-s}.
\end{align}

Using   (v) in   Lemma \ref{use} and   \eqref{equiv11},  we see that
\begin{align}\label{9nov}
|I_3(s)|\le Ce^{-s} \sum_{j=1}^\k    \frac{|\nu_j| }{1-d_j^2}  \iint | \kappa(d_j^*)| \kappa(d^*_i)\rho dy.
\end{align}

Thanks to \eqref{9nov}, \eqref{conmod},   the classical
inequality $ab\le a^2+b^2$,      (ii) in   Lemma \ref{use}, we conclude
\begin{align}\label{9nov1}
|I_3(s)|\le Ce^{-s}.
\end{align}
Combining the  inequalities  \eqref{8nov2},  \eqref{8nov3},  \eqref{8nov4},  \eqref{9nov1}, we infer that
\begin{align}
|\Pi^{d_i^*}_l({0},{\tilde g(q_1))}|\le  C e^{-s}.\label{s2}
\end{align}
\end{itemize}

\medskip

\textbf{Conclusion of the proof of 
\eqref{est:nu}
and \eqref{est:zeta}}:\\
We first project equation \eqref{eq:qq} with $\Pi_0^{d^*_i}$. 
 Using   \eqref{s0}, 
\eqref{s0000},
 \eqref{s1}, \eqref{s2} and \eqref{equiv11}, 
  we conclude
\begin{align}
\displaystyle{\Big|(-1)^{i+1}
d_i'\Pi^{d_i^*}_0
(\partial_d \kappa^*)}+\Pi^{d_i^*}_0(0,R)\Big|
\le  C  \displaystyle{\sum_{j=1, j\neq i}^{\k}\frac{|\nu_j'-\nu_j|+|d'_j|}{1-{d_j}^2}\sqrt{J}}\nonumber\\ 
+C\displaystyle{
\frac{|d_i'|+|\nu_i'-\nu_i|+|\nu_i| }{1-{d_i}^2}\|q\|_{\H}}
+ C \|q\|_{\H}^2 +C \displaystyle{{J}^{1+\delta_1}}+ \displaystyle{\frac{C}{s^{\alpha}}}.\label{vv0}
\end{align}
 By using \eqref{s0000} again, \eqref{conmod},  \eqref{s0}, 
the inequalities \eqref{vv0} and \eqref{equiv11}, we obtain
\begin{align}
\left|\frac {\dot \zeta_i }{c_1} - (e^{-\frac{2}{p-1}  (\zeta_{i}-\zeta_{i-1})} - e^{-\frac{2}{p-1}  (\zeta_{i+1}-\zeta_{i})} )\right|  
\le  C  \displaystyle{\sum_{j=1, j\neq i}^{\k}\frac{|\nu_j'-\nu_j|+|d'_j|}{1-{d_j}^2}\sqrt{J}}\nonumber\\ 
+C\displaystyle{
\frac{|d_i'|+|\nu_i'-\nu_i|+|\nu_i| }{1-{d_i}^2}\|q\|_{\H}}
+ C \|q\|_{\H}^2 +C \displaystyle{{J}^{1+\delta_3}}\label{i1}\\
+CJ\bar J+C|d'_i|\bar J+
 \displaystyle{\frac{C}{s^{\alpha}}},\nonumber
\end{align}
for some  $\delta_3>0$, where $d_i=-\tanh \zeta_i$ and  $J$ and $\bar J$ are defined in \eqref{defJ}.  Recalling the definition \eqref{defJ} of $\zeta_i(s)$    and $J$, we write
%
%
%

\begin{align}
 \frac{|d_i’|}{1-d_i^2}\le  C  \displaystyle{\sum_{j=1, j\neq i}^{\k}\frac{|\nu_j'-\nu_j|+|d'_j|}{1-{d_j}^2}\sqrt{J}}
+C\displaystyle{
\frac{|d_i'|+|\nu_i'-\nu_i|+|\nu_i| }{1-{d_i}^2}\|q\|_{\H}}\nonumber\\ 
+ C \|q\|_{\H}^2 +C \displaystyle{{J}}
+C|d'_i|\bar J+
 \frac{C}{s^{\alpha}},\label{i199}
\end{align}
\medskip

Now, we use the  projection   $\Pi_1^{d^*_i}$.
 Proceeding as for  estimate \eqref{i1}, we use 
  the projections of $\partial_{\nu} \kappa^*(d_j, \nu_j)$ and $\partial_d \kappa^*(d_j, \nu_j)$,
 given in  \eqref{s0000} and the inequality  \eqref{equiv11}
to deduce that
\begin{align}
\displaystyle{\frac{|\nu_i'-\nu_i|}{1-{d_i}^2}}&\le
C \displaystyle{\frac{|d_i'|}{1-{d_i}^2}}+
C  \displaystyle{\sum_{j=1, j\neq i}^{\k}\frac{|\nu_j'-\nu_j|+|d'_j|}{1-{d_j}^2}\sqrt{J}}
\label{i2}\\
&+C\displaystyle{
\frac{|d_i'|+|\nu_i'-\nu_i|+|\nu_i| }{1-{d_i}^2}\|q\|_{\H}}
+ C \|q\|_{\H}^2 + \displaystyle{CJ}+ \displaystyle{\frac{C}{s^{\alpha}}},\nonumber
 \end{align}
hence from \eqref{i199}, 
\begin{align}
\displaystyle{\frac{|\nu_i'-\nu_i|}{1-{d_i}^2}}\le&
C  \displaystyle{\sum_{j=1, j\neq i}^{\k}\frac{|\nu_j'-\nu_j|+|d'_j|}{1-{d_j}^2}\sqrt{J}}
+C\displaystyle{
\frac{|d_i'|+|\nu_i'-\nu_i|+|\nu_i| }{1-{d_i}^2}\|q\|_{\H}}\label{i222}\\
&+ C \|q\|_{\H}^2 + \displaystyle{CJ}+C|d'_i|\bar J+ \displaystyle{\frac{C}{s^{\alpha}}}.\nonumber
 \end{align}
Summing-up the estimates \eqref{i199} and \eqref{i222}  in $i$,   using   the smallness of $\|q\|_{\H}$, $\bar J$  and $J$, for $s_0$ large enough (see \eqref{conmod}), we  conclude 
\begin{align}
\frac{| \dot \nu_i - \nu_i |}{1-d_i^2} +
 \displaystyle{\frac{|d_i'|}{1-{d_i}^2}}
& \le C  \left( \| q \|_{\q H}^2 + J + \| q \|_{\q H} \bar J\right)+\frac{C}{s^{\alpha}}.
\label{i3}
\end{align}
Clearly, using  \eqref{i1}  and \eqref{i3},  we conclude \eqref{est:nu} and \eqref{est:zeta}.

\bigskip

{\bf Part 2: A Lyapunov functional for equation \eqref{eq:qq}}\\
Let us assume that $s_0\ge 1$ is  large enough  and take $s\in [s_0, \bar s)$.

\medskip

We now  prove estimate \eqref{est:q}.
Like for Claim 4.8 page 2867 in \cite{MZisol10}, the idea is  to  construct a Lyapunov functional for equation \eqref{eq:qq} which is equivalent to the norm squared.
Let us introduce for all $s\in [s_0,\bar s]$,
\begin{equation}\label{defE12}
\begin{array}{rcl}
E_1(s) &=& \frac 12 \varphi(q,q)+R_-(s),\\
E_2(s) &=&E_1(s)
+\eta \iint q_1 q_2 \rho dy,
\end{array}
\end{equation}
where $\eta \in (0,1]$
 will be  fixed at the the end
of the proof of \eqref{est:q}
and  $ R_-$, $\varphi$, $\q H(q_1)$  and $ \q F(q_1)$ are given by:
\begin{eqnarray}
R_-(s)&=&-\iint\q  H(q_1)\rho dy-e^{-\frac{2(p+1)s}{p-1}}\iint\q  F(e^{\frac{2s}{p-1}}(K^*_1+q_1))\rho dy,
\label{defhf}\\
\varphi(r,q) &=&\displaystyle{ \iint \left(r_1'q_1' (1-y^2)-\psi r_1q_1+r_2q_2\right)\rho dy,}
\label{defhf1}\\
\q H(q_1)& = &\displaystyle{\int_0^{q_1}h(\xi) d\xi = \frac{|K^*_1+q_1|^{p+1}}{p+1}-\frac{{|K^*_1|}^{p+1}}{p+1}-{|K^*_1|}^{p-1}K^*_1 q_1 - \frac p2 {|K^*_1|}^{p-1}q_1^2, }\nonumber\\
\nonumber\\
\q F(q_1)& =&\displaystyle{ \int_0^{q_1}\tilde f(\xi) d\xi},\nonumber
\end{eqnarray}
where $\psi(d,y)$, $  h(q_1)$, $\tilde f(q_1)$ and $K^*_1$ are defined in \eqref{vd2}, \eqref{vd3}, \eqref{vd4}, 
respectively.


\medskip
Before starting  the proof of  \eqref{est:q},  let us give in  the following lemma  some useful estimates.
\begin{lem}\label{lemquad}
For all $s\in [s_0, \bar s)$, we have:
\begin{align}
C^{-1}  \|q\|_{\H}^2\le  \varphi(q,q)\le  C\|q\|_{\H}^2,
\label{eq2}\\
\left|\iint \q H(q_1) \rho dy\right|\le C \|q\|_{\H}^{1+\min(p,2)}\le 
 C s_0^{(1-\min(p,2))/2}\|q\|_{\H}^2,\label{eq2a}\\
\left|
e^{-\frac{2(p+1)s}{p-1}}\iint \q   F(e^{\frac{2s}{p-1}}(K^*_1+q_1))
\rho dy\right|\le \displaystyle{\frac{C}{s^{\alpha}}}, \label{eq2b}\\
 \Big|\iint (K^*_1+q_1) \tilde f(q_1)
\rho dy\Big| \le \displaystyle{
\frac{C}{s^{\alpha}}}. \label{eq2d}
\end{align}
\end{lem}
\begin{proof} {}~ \\
-  According to  \eqref{conmod} and \eqref{equiv11},  we know that    $\sum_{i=1}^k \frac{|\nu_i|}{1-|d_i^*|}$ and  $ \|q\|_{\H}$  are   small enough, for $s_0$  large enough. Then
we can adapt with no difficulty the proof given in \cite{MZisol10} (page 2898) to deduce the estimates   \eqref{eq2} and \eqref{eq2a}.\\
- 
Clearly,  using    similar arguments to the proof of  \eqref{s1}, we prove   the following estimate:
\begin{equation}
 \Big|e^{-\frac{2(p+1)s}{p-1}} \q   F(e^{\frac{2s}{p-1}}(K^*_1+q_1))|
\le Ce^{-\frac{p+1}{p-1}s}+\frac{C}{s^{\alpha}}( |K^*_1|^{p+1}+ |q_1|^{p+1}).\label{vF}
\end{equation}
Thanks to   (ii) and (iv)  in   Lemma \ref{use} and  \eqref{conmod},    we conclude
 \begin{equation}\label{9nov5}
\iint |K^*_1|^{p+1}\rho dy\le C\sum_{i=1}^k \iint   |\kappa (d^*_i)|^{p+1}\rho dy\le C.
\end{equation}
By integrating  inequality \eqref{vF} over $(-1,1)$,  and taking into account \eqref{9nov5},  \eqref{conmod} and  (i)   in   Lemma \ref{use},    we  obtain the estimate \eqref{eq2b}.\\
- Proceeding  similarly  as for   \eqref{eq2b}  and using  estimates \eqref{conmod} and \eqref{equiv11}, we  easily deduce \eqref{eq2d}. This concludes the proof of
Lemma \ref{lemquad}.
\end{proof}

%

\medskip

In order to  construct a Lyapunov functional for equation \eqref{eq:qq}, 
we  need  to prove the following estimates:
\begin{lem}\label{lemproj*} For all $s\in [s_0, \bar s)$, we have:\\
{\bf (i)(Control of the time derivative of $E_1$)}
\begin{align}\label{anas5}
\frac{d}{ds} E_1 ( s) \le  -\displaystyle{\frac 2{p-1}}\iint q_2^2
\frac \rho{1-y^2} dy 
+ C  ( \| q \|_{\q H} +\bar J)
\varphi(q,q)+CJ \sqrt{\varphi(q,q)}
 + C J^{\bar p}+\frac{C}{s^{\alpha}} .
\end{align}
%
%
{\bf(ii) (Control of the time derivative of   $\iint q_1q_2 \rho dy$)}
\begin{align}\label{v17}
\frac d{ds}\iint q_1q_2 \rho dy \le&
-\frac 7{10}\varphi(q,q) +C \iint q_{2}^2 \frac \rho{1-y^2}dy+CJ \sqrt{\varphi(q,q)}+ CJ^{\bar p}+\frac{C}{s^{\alpha}}.
\end{align}
\end{lem}
\begin{proof}{}~ \\
{\bf{Proof of (i) (Control of the time derivative of the terms of $E_1$)}.}\\
 Using the  definitions \eqref{defhf1}   of  $\varphi$ and \eqref{vd2}   of  $\psi$, we write
\begin{equation}\label{anas}
\frac12\frac{d}{ds} \varphi ( q, q) =
 \varphi (\partial_s q, q))+I_4(s),
\end{equation}
where
\begin{equation}\label{v13j}
I_4(s)= \frac{p(p-1)}2\sum_{j=1}^\k  (-1)^{j+1}\iint
\Big(  d_j' \partial_d \kappa^*_1 +\nu_j' \partial_{\nu} \kappa^*_1\Big)
 |K^*_1|^{p-3}K^*_1 q_1^2 \rho dy.
\end{equation}
 Using equation \eqref{eq:qq},  we write
\begin{eqnarray}
\varphi(\partial_s q, q)&=& \varphi(\ll q,q)\underbrace{-\sum_{j=1}^\k (-1)^j[ (\nu_j-\nu_j') \varphi(\pnu \kappa^*,q)
+d_j' \varphi(\partial_d \kappa^*,q)]}_{I_5(s)}\nonumber\\
&&+ \varphi((0,R),q)+
\iint  q_2
 \tilde{g}(q_1) \rho dy+
\iint  q_2
\Big( h(q_1)+ \tilde{f}(q_1)\Big) \rho dy.\label{mohsen}
\end{eqnarray}

\medskip

In the remaining part of the  proof, we need   some estimates  proved    in  \cite[Appendix C]{MZisol10}. 
For that reason, we will recall  the  following  estimates   (C.24) and   (C.26)
  from that paper and   (C.39) in  \cite[Appendix C]{MZajm11},
true under hypothesis \eqref{conmod}
\begin{equation}\label{estR}
\begin{array}{ll}
\varphi(\ll q,q) = -\displaystyle{\frac 4{p-1}}\iint q_2^2\frac \rho{1-y^2} dy,&
\\
\iint \kappa(d_i^*,y)|f(q_1)|\rho dy \le C\|q\|_{\H}^2,&
|R_-|\le C\|q\|_{\H}^{\bar p+1},\\
\Big| \iint \partial_d \kappa^*_1(d_i, \nu_i) |K^*_1|^{p-3}K^*_1 q_1^2 \rho dy\Big|\le &\displaystyle{\frac C{1-{d_i^*}^2}} \|q\|_{\H}^2,
\\
\Big| \iint \partial_{\nu} \kappa^*_1(d_i, \nu_i) |K^*_1|^{p-3}K^*_1 q_1^2 \rho dy\Big|\le &\displaystyle{\frac C{1-{d_i^*}^2}} \|q\|_{\H}^2,
\\
 \iint R^2  (1-y^2)\rho dy\le   C  (\hat J^*)^{2},&
\end{array}
\end{equation}
where  $\bar p$ is defined \eqref{defpb}.
Since we have from    \eqref{equiv11} and   \eqref{equiv112} $ \hat J^*\le  C( J^*)^{\frac{\bar p}2} \le     C J^{\frac{\bar p}2},$
 we get from \eqref{estR}
\begin{eqnarray}
\varphi((0,R),q) 
= \iint R q_2 \rho dy &\le & \frac 1{p-1} \iint q_2^2 \frac \rho{1-y^2} dy + C \iint R^2  (1-y^2)\rho dy\nonumber\\
&\le & \frac 1{p-1}  \iint q_2^2 \frac \rho{1-y^2} dy + C J^{\bar p}.\label{estRR}
\end{eqnarray}
Using \eqref{estR},  \eqref{eq2}, \eqref{conmod}, \eqref{equiv11}, \eqref{est:nu} and \eqref{est:zeta},   we write
\begin{equation}
|I_4(s)|+|I_5(s)|\le   C   \| q \|_{\q H}^3 + C\| q \|_{\q H}  J+ C\| q \|^2_{\q H} \bar J +
\frac{C}{s^{\alpha}}.\label{v10}
\end{equation}
Moreover, by using 
\eqref{8novbis}, we write
\begin{align}
\big|\iint  q_2
 \tilde{g}(q_1)\rho dy\big|  
\le \ \ \ \ \ & {   Ce^{-s}\iint 
  \big(|\partial_yq_1|+|q_2|+|q_1|+|q_1|^{\frac{p+1}2}\big)|q_2| \rho dy}\nonumber\\
&+{
Ce^{-s}\iint  \big(|\partial_yK^*_1|+|K^*_1|+|K^*_1|^{\frac{p+1}2}\big)|q_2| \rho dy}\label{10nov}\\
&+{Ce^{-s} 
\sum_{j=1}^\k  | \nu_j| \iint
|\partial_{\nu} \kappa^*_1(d_j, \nu_j)| \ |q_2| \rho dy}+
Ce^{-\frac{2ps}{p-1}}.\nonumber
\end{align}
Proceeding similarly as for
  $I_1(s)$,   $I_2(s)$ and   $I_3(s)$ defined in \eqref{8nov2},  we get
\begin{align}
\big|\iint  q_2
 \tilde{g}(q_1)\rho dy\big|    \leq  C e^{-s} \iint q_2^2\frac{\rho}{ 1-y^2} dy
+C e^{-s}.\label{v11}
\end{align}
Gathering the estimates  \eqref{anas},  \eqref{mohsen},  \eqref{estR},  \eqref{estRR}, 
 \eqref{v10} and
 \eqref{v11}, we get for $s_0$ large enough,
\begin{align}\label{anas1}
\frac12\frac{d}{ds} \varphi ( q, q) \le&  -\displaystyle{\frac 2{p-1}}\iint q_2^2
\frac \rho{1-y^2} dy+ \iint q_2\Big(h(q_1)+ \tilde{f}(q_1)\Big)\rho  dy \\
&+ C   \| q \|_{\q H}^3 + C\| q \|_{\q H}  J+ C\| q \|^2_{\q H} \bar J+ C J^{\bar p}+\frac{C}{s^{\alpha}}.\nonumber
\end{align}

\medskip

Using the definition of $R_-,$ given in \eqref{defhf}, we write 
\begin{eqnarray}\label{R1}
R_-'(s)&=&- \iint  \Big(\partial_s q_1+\partial_sK^*_1\Big)
\Big( h(q_1)+ \tilde{f}(q_1)\Big) \rho dy-\frac{2}{p-1} \iint  (q_1+K^*_1)\tilde{f}(q_1) \rho dy\nonumber\\
&&+\frac{2(p+1)}{p-1}e^{-\frac{2(p+1)s}{p-1}}
\iint \q   F(e^{\frac{2s}{p-1}}(K^*_1+q_1))
\rho dy-I_4(s),
\end{eqnarray}
where $I_4(s)$ is defined in \eqref{v13j}. By exploiting  the identity \eqref{q21},  identity \eqref{R1} becomes 
\begin{eqnarray}\label{R2}
R_-'(s)&=&- \iint  q_2
\Big( h(q_1)+ \tilde{f}(q_1)\Big) \rho dy\underbrace{-\frac{2}{p-1} \iint  (q_1+K^*_1)\tilde{f}(q_1) \rho dy}_{I_6(s)}\nonumber\\
&&+\underbrace{\frac{2(p+1)}{p-1}e^{-\frac{2(p+1)s}{p-1}}
\iint \q   F(e^{\frac{2s}{p-1}}(K^*_1+q_1))
\rho dy}_{I_7(s)}
\\
&&+ \underbrace{\sum_{j=1}^\k  (-1)^j \nu_j \iint\partial_{\nu} \kappa^*_1\big[ h(q_1)
+ \tilde{f}(q_1)\big]  \rho dy}_{I_8(s)}-I_4(s).\nonumber
\end{eqnarray} 
Thanks to \eqref{eq2b}, \eqref{eq2d} and \eqref{conmod}, we write
\begin{align}
|I_6(s)|+|I_7(s)|\le 
  \displaystyle{\frac{C}{s^{\alpha}}.} \label{eqvd}
\end{align}
Finally, it remains only to control the term  $I_8(s)$. 
By exploiting   (v) in   Lemma \ref{use},  we write
\begin{align}
|I_8(s)|\le C
  \sum_{j=1}^\k  \frac{|\nu_j|}{1-|d_j^*|} \iint  \kappa (d_j^*) \Big( |h(q_1)|
+ |\tilde{f}(q_1)|\Big)  \rho dy.
\end{align}
Similarly  as for   \eqref{equiv11},
\eqref{s1}, \eqref{s0} and \eqref{conmod}, we easily  prove that  
\begin{align}\label{eqvd1}
|I_8(s)|\le C  \sum_{j=1}^\k  \frac{|\nu_j|}{1-|d_j^*|} \Big(   \|q\|_{\H}^{2}
+\frac{C}{s^{\alpha}}     \Big)\le C   \|q\|_{\H}^{2}\bar  J 
+\frac{C}{s^{\alpha}}    .
\end{align}
Gathering the estimates \eqref{R2}, \eqref{eqvd}, \eqref{eqvd1}, \eqref{v10}, \eqref{conmod} and 
\eqref{defJ},
we get 
\begin{align}\label{eq0}
R_-'(s)\le- \iint  q_2
\Big( h(q_1)+ \tilde{f}(q_1)\Big) \rho dy
+ C   \| q \|_{\q H}^3 + C\| q \|_{\q H}  J+ C\| q \|^2_{\q H} \bar J +\frac{C}{s^{\alpha}}.
\end{align} 
Thanks to \eqref{anas1},  \eqref{eq0}  and  
 the definition of $E_1(s),$ given in 
\eqref{defE12}, we clearly have
\begin{align}\label{anas05}
\frac{d}{ds} E_1 ( s) \le  -\displaystyle{\frac 2{p-1}}\iint q_2^2
\frac \rho{1-y^2} dy 
+ C  \Big( \| q \|_{\q H}^3 + \| q \|_{\q H}  J+ \| q \|^2_{\q H} \bar J+  J^{\bar p}\Big)+\frac{C}{s^{\alpha}} .
\end{align}
According to  \eqref{conmod},  we can see that    $ \|q\|_{\H}$  can be   made    small enough provided that $s_0$ is large enough.
As a consequence of      \eqref{eq2} and \eqref{anas05},   we get the desired estimate \eqref{anas5}. 
This ends
the proof of (i) in Lemma \ref{lemproj*}.

\bigskip

{\bf Proof of (ii): Control of the time derivative of   $\iint q_1q_2 \rho dy$}

Using  equation \eqref{eq:qq}  and the identity  \eqref{q21}, we write
\begin{eqnarray*}
&&\frac d{ds} \iint q_1 q_2 \rho dy = \iint \ps q_1 q_2 \rho dy + \iint \ps q_2 q_1 \rho dy\\
&=& \underbrace{- \sum_{j=1}^\k(-1)^j\left(d_j'\iint \partial_{d} \kappa^* \cdot (q_2,q_1)\rho dy + (\nu_j'- \nu_j)\iint \partial_{\nu} \kappa^*\cdot (q_2, q_1) \rho dy\right)}_{I_9(s)}\\
&&+\iint q_2^2 \rho dy + \iint q_1 (  \q L  q_1 + \psi q_1 - \frac{p+3}{p-1} q_2 - 2y \py q_2 +h(q_1)+\tilde f(q_1)+\tilde 
g(q_1)+R)\rho dy,
\end{eqnarray*}
where the dot ``$\cdot$'' stands for the usual inner product  in $\R^n$. 
 
Let us first recall the following  estimates proved in pages 2904 and 2905 of 
\cite{MZisol10}: 
\[
\begin{array}{l}
\iint q_2^2 \rho dy \le \iint q_2^2 \frac \rho{1-y^2} dy,\;\;\iint q_1( \q L  q_1 + \psi q_1)\rho dy \le - \varphi (q,q) + \iint q_2^2 \frac \rho{1-y^2} dy,\\
\\
\left|- \frac{p+3}{p-1} \iint q_1 q_2 \rho dy - 2 \iint q_1 \py q_2 \rho dy\right|\le \frac{1}{10}  \varphi (q,q)+C \iint q_2^2 \frac \rho{1-y^2} dy,\\
\\
\left|\iint q_1 h(q_1)\rho dy\right| \le \frac{1}{10} \varphi (q,q).
\end{array}
\]
Furthermore, by using  the Hardy Sobolev inequality stated in item (i) of Lemma \ref{use}, \eqref{eq2},  \eqref{estR} and \eqref{equiv11}, we conclude 
$$\iint q_1R\rho dy \le \frac{1}{10}\varphi (q,q) + C J^{\bar p}.$$
Moreover, arguing as in \eqref{s1} and \eqref{s2}, we obtain 
\[
\begin{array}{l}
\left|\iint q_1 \tilde f(q_1)\rho dy\right| \le  \displaystyle{\frac{C}{s^{\alpha}}},\qquad 
\left|\iint q_1 \tilde g(q_1)\rho dy\right|\le  
Ce^{-s}.
\end{array}
\]
Note also that,    from   (ii) and   
 (v) in   Lemma \ref{use},
     \eqref{conmod}, \eqref{equiv11},
the Cauchy-Schwarz inequality, we write
\begin{align}\label{oct1}
\big|I_9(s)\big|\le C\sum_{j=1}^{k} \frac{|d_j'|+|\nu_j'- \nu_j|}{1-d_j^2}\|q\|_{\H}.
\end{align}
By exploiting \eqref{oct1},  \eqref{est:nu} and  \eqref{est:zeta}, 
we deduce that
\begin{align}
\big|I_9(s)\big|
\le  C \Big(  \| q \|_{\q H}^3 + \| q \|_{\q H}  J+ \| q \|^2_{\q H} \bar J+\frac{1}{s^{\alpha}}\Big) .
\end{align}
According to  \eqref{conmod},  we can see that  $\bar J$ and  $ \|q\|_{\H}$  can be   made small enough provided that $s_0$ is large enough. Consequently, collecting
 the above estimates   and \eqref{eq2}, we  obtain the desired estimate \eqref{v17}, and this ends
the proof of (ii) in Lemma \ref{lemproj*}.

\bigskip

\end{proof}
%
Now, we are able to prove \eqref{est:q}.\\
{\bf Proof of  \eqref{est:q}:}
Let $\eta \in (0,1)$. 
By exploiting \eqref{anas5} and \eqref{v17} and  the definition of $E_2(s),$ given in \eqref{defE12}, we have
\begin{align}\label{v18}
\frac{d}{ds} E_2 ( s)+\frac {\eta}2  E_2(s) \le -\Big(\frac {\eta}{5}- C  ( \| q \|_{\q H} +\bar J)\Big)\varphi(q,q)+\frac{\eta^2}2
\iint q_1q_2 \rho dy+\frac{\eta}2R_-(s) \nonumber\\
 -(\frac 2{p-1}-C\eta )\iint q_2^2
\frac \rho{1-y^2} dy+CJ\sqrt{\varphi(q,q)}
 + C J^{\bar p}+\frac{C}{s^{\alpha}}.
\end{align}
Observing   inequalities  \eqref{eq2a},  \eqref{eq2b},  \eqref{eq2}, \eqref{conmod} along with the Cauchy-Schwarz inequality
 yield 
\begin{equation}\label{v180}
\frac{\eta^2}2
\iint q_1q_2 \rho+\frac{\eta}2R_-(s) \le C \eta^2\varphi(q,q)+
 C\eta  s_0^{(1-\min(p,2))/2}\varphi(q,q)
+ \frac{C}{s^{\alpha}}.
\end{equation}
Moreover,  by \eqref{conmod},  we easily deduce that 
\begin{equation}\label{vvvv180}
CJ\sqrt{\varphi(q,q)}\le  \frac {\eta}{10}\varphi(q,q)+\frac{C}{\eta} J^2\le  \frac {\eta}{10}\varphi(q,q)+\frac{C}{\eta} J^{\bar p}.
\end{equation}
From  \eqref{v18},  \eqref{v180}, 
we get
\begin{align}\label{v181}
\frac{d}{ds} E_2 ( s)+\frac {\eta}2  E_2(s) \le& -\Big(\frac {\eta}{5}- C\eta  s_0^{(1-\min(p,2))/2}- C \eta^2- C  ( \| q \|_{\q H} +\bar J)\Big)\varphi(q,q) \nonumber\\
& -(\frac 2{p-1}-C\eta )\iint q_2^2
\frac \rho{1-y^2} dy
 + \frac{C}{\eta} J^{\bar p}+\frac{C}{s^{\alpha}}.
\end{align}
Note that, once again,  due to the fact that
 $s_0$ is large enough, and according to  \eqref{conmod},  we can consider   $ \|q\|_{\H}$,  $ s_0^{(1-\min(p,2))/2}$ and $\bar J$  as   small terms.  Then,
  there exists   $\eta_0>0$  such that, for all $\eta \in (0,\eta_0]$, we have 
\begin{align}\label{v181}
\frac{d}{ds} E_2 ( s)+\frac {\eta}2  E_2(s) \le
   \frac{C}{\eta} J^{\bar p}+\frac{C}{s^{\alpha}}.
\end{align}
 By definition \eqref{defJ} of $J$, we write
\begin{equation}\label{k11}
|(J^{\bar p})'| = \frac{ 2\bar p}{p-1} \left|\sum_{j=1}^{\k-1}({\zeta_{j+1}'}-{\zeta_j'})e^{- \frac 2{p-1}(\zeta_{j+1}-\zeta_j)}\right|J^{{\bar p -1}}
\le \frac{4\bar pJ^{\bar p}}{p-1}\sum_{j=1}^\k |{\zeta_j'}|.
\end{equation}
Exploiting  
 \eqref{est:zeta} and    \eqref{conmod}, we conclude that
$|{\zeta'_j}| \le Cs_0^{-\frac12}. $ Consequently, we have
\begin{equation}\label{k3}
\Big|(J^{\bar p})'\Big|\le C s_0^{-\frac12} J^{\bar p}.
\end{equation}
Since $s_0$ is large enough, then we can write 
\begin{equation}\label{k3}
\Big|(J^{\bar p})'\Big|\le \frac{\eta}{4} J^{\bar p}.
\end{equation}
Let us introduce
\begin{equation}\label{v190}
E_3(s)=  E_2(s)-\frac1{\eta^3}
  J^{\bar p}.
\end{equation} 
By using    \eqref{v181},  \eqref{k3} and \eqref{v190},  we are now  in a position   to get,  
\begin{align}\label{v182}
\frac{d}{ds} E_3 ( s)+\frac {\eta}2  E_3(s) \le
  \frac1{\eta^2}(C\eta
-\frac{1}{4}) J^{\bar p}+\frac{C}{s^{\alpha}}, \quad \forall s\ge s_0, \quad \forall \eta \in (0,\eta_0].
\end{align}
Then  for 
 $\eta_0$ small enough,  we have
\begin{align}\label{v183}
\frac{d}{ds} E_3 ( s)+\frac {\eta}2  E_3(s) \le 
\frac{C}{s^{\alpha}}, 
 \quad \forall s\ge s_0, \quad \forall \eta \in (0,\eta_0].
\end{align}
Integrating the last  inequality and using the definition of $E_3(s)$, we conclude
\begin{equation}\label{v184}
E_2(s)\le E_2(s_0)  e^{-\frac {\eta (s-s_0)}2}+\frac1{\eta^3}
  J^{\bar p}+ \frac{C}{ \eta s^{\alpha}},   \quad \forall s\ge s_0 \quad \forall \eta \in (0,\eta_0].
\end{equation}
 Using the definition of $E_2(s)$ as in \eqref{defE12},  the estimates \eqref{eq2},   \eqref{eq2a},  \eqref{eq2b}  and \eqref{v180}, we conclude
\begin{align*}\label{v184}
\phi(q,q)\le&\ \  C \| q(s_0) \|_{\q H}^2  \ e^{-\frac {\eta (s-s_0)}2}+\frac1{\eta^3}
  J^{\bar p}+ \frac{C}{\eta  s^{\alpha}}\\
&+\big(C\eta 
 + C s_0^{(1-\min(p,2))/2}\big)\phi(q,q),  \quad \forall s\ge s_0,  \quad \forall \eta \in (0,\eta_0].
\end{align*}
The above  inequality, used for  $\eta$ small enough,  together with \eqref{eq2}  imply \eqref{est:q}. This concludes  the proof of Proposition \ref{propdyn}. 
\end{proof}

{\bf{Acknowledgements.}} The first author  would like to thank   the Laboratoire Analyse
G\' eom\' etrie et Applications (LAGA) of  the University Paris 13, where he   was invited as a Visiting Professor 
in June 2018 and where part of this work was done. He is grateful to all the LAGA Faculty members
for the hospitality and the stimulating atmosphere.

\def\cprime{$'$} 

\noindent{\bf Address}:\\
  Imam Abdulrahman Bin Faisal University
P.O. Box 1982 Dammam, Saudi Arabia.\\
\vspace{-7mm}
\begin{verbatim}
e-mail:  mahamza@iau.edu.sa
\end{verbatim}
Universit\'e Paris 13, Institut Galil\'ee,
Laboratoire Analyse, G\'eom\'etrie et Applications, CNRS UMR 7539,
99 avenue J.B. Cl\'ement, 93430 Villetaneuse, France.\\
\vspace{-7mm}
\begin{verbatim}
e-mail: Hatem.Zaag@univ-paris13.fr
\end{verbatim}

\end{document}